\documentclass[a4paper,oneside,11pt]{amsart}
\usepackage{amssymb}
\usepackage{amsmath}
\usepackage{thmtools}
\usepackage{textcomp}
\usepackage{thm-restate}
\usepackage{mathrsfs}
\usepackage[hidelinks]{hyperref}
\usepackage{hyperref}
\usepackage{cleveref}
\usepackage{enumitem}
\usepackage{cleveref}
\renewcommand\labelenumi{(\roman{enumi})}
\renewcommand\theenumi\labelenumi

\newtheorem{lem}{Lemma}
\newtheorem{prop}{Proposition}
\newtheorem{thm}{Theorem}

\newtheorem{cor}{Corollary}
\theoremstyle{definition}
\newtheorem{defn}{Definition}
\newtheorem{rem}{Remark}

\newtheorem*{problem}{Problem}
\newcounter{numl}
\newcommand{\labelnuml}{\textup{(\roman{numl})}}
\newenvironment{numlist}{\begin{list}{\labelnuml}%
{\usecounter{numl}\setlength{\leftmargin}{0pt}%
\setlength{\itemindent}{2\parindent}%
\setlength{\itemsep}{\smallskipamount}\def
\makelabel ##1{\hss \llap {\upshape ##1}}}}{\end{list}}

\DeclareSymbolFont{script}{U}{eus}{m}{n}
\DeclareSymbolFontAlphabet{\mathscr}{script}
\DeclareMathSymbol{\Wedge}{0}{script}{"5E}
\DeclareMathAlphabet{\mathrmsl}{OT1}{cmr}{m}{sl}
\newcommand{\R}{{\mathbb R}}
\newcommand{\C}{{\mathbb C}}

\newcommand{\T}{{\mathbb T}}

\newcommand{\cO}{{\mathcal O}}
\newcommand{\Id}{\mathit{Id}}
\newcommand{\sub}{\subseteq}

\newcommand{\grad}{\mathrm{grad}}

\makeatletter
\newcommand{\leqnos}{\tagsleft@true\let\veqno\@@leqno}
\newcommand{\reqnos}{\tagsleft@false\let\veqno\@@eqno}
\leqnos
\makeatother

\usepackage[left=2.5cm,top=3cm,right=2.5cm,bottom=3cm,bindingoffset=0.5cm]{geometry}

\makeatletter
\let\@@citation@@=\citation

\renewcommand{\citation}[1]{\@@citation@@{#1}%
\@for\@tempa:=#1\do{\@ifundefined{cit@\@tempa}%
  {\global\@namedef{cit@\@tempa}{}}{}}%
}

\def\@lbibitem[#1]#2#3\par{%
  \@ifundefined{cit@#2}{}{\item[\@biblabel{#1}\hfill]}%
  \if@filesw
      {\let\protect\noexpand
       \immediate
       \write\@auxout{\string\bibcite{#2}{#1}}}\fi\ignorespaces
  \@ifundefined{cit@#2}{}{#3}}
\def\@bibitem#1#2\par{%
  \@ifundefined{cit@#1}{}{\item}%
  \if@filesw \immediate\write\@auxout
    {\string\bibcite{#1}{\the\value{\@listctr}}}\fi\ignorespaces
  \@ifundefined{cit@#1}{}{#2}}
\makeatother

\begin{document}

\author[A. Lahdili]{Abdellah Lahdili}
\address{Lahdili Abdellah \\ D{\'e}partement de Math{\'e}matiques\\
UQAM\\ C.P. 8888 \\ Succursale Centre-ville \\ Montr{\'e}al (Qu{\'e}bec) \\
H3C 3P8 \\ Canada}
\email{lahdili.abdellah@courrier.uqam.ca}

\title[]{Automorphisms and deformations of conformally K\"ahler, Einstein--Maxwell metrics}

\begin{abstract} We obtain a structure theorem for the group of holomorphic automorphisms of a conformally K\"ahler, Einstein--Maxwell metric, extending the classical results of Matsushima~\cite{M}, Licherowicz~\cite{L} and Calabi~\cite{calabi} in the K\"ahler--Einstein, cscK, and extremal K\"ahler cases. Combined with previous results of LeBrun~\cite{LeB1}, Apostolov--Maschler~\cite{AM} and Futaki--Ono~\cite{FO}, this completes the classification of the conformally K\"ahler, Einstein--Maxwell metrics on $\mathbb{{CP}}^1 \times \mathbb{{CP}}^1$. We also use our result in order to introduce a (relative) Mabuchi energy in the more general context of $(K, q, a)$-extremal K\"ahler metrics in a given K\"ahler class, and show that the existence of $(K, q, a)$-extremal K\"ahler metrics is stable under small deformation of the K\"ahler class, the Killing vector field $K$ and the normalization constant $a$.
\end{abstract}

\maketitle

\section{Introduction}  Let $(M, J)$ be a compact K\"ahler manifold of complex dimension $m$ and $g$ a $J$-compatible K\"ahler metric. Following \cite{AM}, the Hermitian metric ${\tilde g}= \frac{1}{f^2} g$ is said to be {\it conformally K\"ahler, Einstein--Maxwell} if $\tilde g$ has
\begin{enumerate}
\item[(a)] 
$J$-invariant Ricci tensor, i.e. 
\begin{equation}\label{Ric J-inv}
{\rm Ric}_{\tilde g}(\cdot, \cdot ) = {\rm Ric}_{\tilde g}(J\cdot, J\cdot),
\end{equation}
\item[(b)] constant scalar curvature, i.e. 
\begin{equation}
{\rm Scal}_{\tilde g}=const.
\end{equation}
\end{enumerate}
These conditions extend to higher dimensions a $4$-dimensional riemannian signature analogue of the Einstein--Maxwell equations in General Relativity, see~\cite{LeB0, ambitoric1}.  

In \cite{AM}, Apostolov--Maschler initiated a study of conformally K\"ahler, Einstein--Maxwell K\"ahler metrics in a framework similar to the famous Calabi problem~\cite{calabi} of finding extremal K\"ahler metrics in a given K\"ahler class, and set the existence problem of the conformally K\"ahler, Einstein--Maxwell K\"ahler metrics in a formal GIT picture, extending the work of Donaldson and Fujiki~\cite{donaldson,fujiki} characterizing the Calabi extremal metrics as critical points of the norm of the corresponding  moment map. In particular, fixing a K\"ahler class $\Omega$ on $(M, J)$, a quasi-periodic real holomorphic vector field $K$ with zeroes, and a real positive constant $a>0$, it was shown in \cite{AM} that there is a natural obstruction to the existence of conformally K\"ahler, Einstein--Maxwell K\"ahler metrics associated to the above data, similar to the Futaki invariant~\cite{Futaki1,Futaki2} in the K\"ahler--Einstein and the constant scalar curvature K\"ahler (cscK) cases. More recently, Futaki--Ono~\cite{FO} have characterized the latter obstruction in terms of a {\it volume-minimizing} condition on $K$, reminiscent to the constant scalar curvature Sasaki case~\cite{MSY,FOW,legendre}.

\smallskip 
The purpose of this paper is to extend two fundamental results in the theory of extremal K\"ahler metrics to a more general context relevant the conformally K\"ahler, Einstein--Maxwell metrics described above.  
The first result is a suitable extension of Calabi's Theorem~\cite{calabi} on the structure of the group of holomorphic automorhisms of a compact extremal K\"ahler manifold. To state it, let  $g$ be a K\"ahler metric on $(M, J)$ endowed with a Killing vector field $K$ with zeroes. Hodge theory implies (see e.g. ~\cite{Gauduchon-book}) that $K$ is hamiltonian with respect to the K\"ahler form $\omega = gJ$, i.e.  $\imath_{K} \omega = - df$ for a smooth function on $M$, called a {\it Killing potential} of $K$. We normalize $f=f_{(K,\omega,a)}$ by requiring $\int_M f_{(K,\omega,a)} dv_{g} = a >0,$ where the positive real constant $a$ is such that $f_{(K,\omega,a)} >0$ on $M$. Then, for any fixed real number $q$ we define the {\it $(K, q, a)$-scalar curvature} of $g$ to be
\begin{equation}\label{(K,nu,a)-scalar-curvature}
{\rm S}_{(K, q, a)}(g):= f^2_{(K,\omega,a)} {\rm Scal}_g  + 2q f_{(K, \omega, a)} \Delta_g \left(f_{(K,\omega, a)}\right) - q(q-1)|K|^2_g,
\end{equation}
where ${\rm Scal}_g$ denotes the usual scalar curvature, $|\cdot|_g$ is the tensor norm induced by $g$, and $\Delta_g$ stands for the riemannian Laplacian on functions.\\

\smallskip

The point of this definition is that the condition \eqref{Ric J-inv} above yields that the conformal factor $f$ is a positive Killing potential of a Killing vector field $K$ for $g$, whereas the scalar curvature ${\rm Scal}_{\tilde g}$ of $\tilde g = \frac{1}{f^2} g$ is given by the formula \eqref{(K,nu,a)-scalar-curvature} with $q= -(2m-1)$. Other choices of the {\it weight} $q$ lead to other interesting geometric problems, as it was observed in \cite{ACGL}. We also notice that the GIT framework of \cite{AM} makes sense for any choice of the weight $q$ as above, see Section~\ref{s1} below.

\begin{defn}\label{d:(K,nu.a)-extremal} Let $g$ be a K\"ahler metric on $(M, J)$ endowed with a Killing vector field $K$ as above, and  $a>0$ a real constant such the corresponding Killing potential $f_{(K,\omega,a)}>0$ on $M$. We say that $g$ is $(K, q, a)$-extremal if its $(K, q, a)$-scalar curvature given by \eqref{(K,nu,a)-scalar-curvature} is a Killing potential, i.e. $\Xi=J {\rm grad}_g({\rm Scal}_{(K, q, a)}(g))$ is a Killing vector field for $g$.
\end{defn}

The definition above incorporates the case when ${\rm S}_{(K, q, a)}(g)$ is constant, which in turn links to the initial motivation of studying conformally K\"ahler, Einstein--Maxwell metrics. We denote by $\mathfrak{h}^{K}$ (resp. $\mathfrak{k}^{K}$) the centralizer of $K$ in the Lie algebra of holomorphic vector fields (resp. Killing vector fields) of $(M,J)$ (resp. $(M,g)$) and ${\rm Aut}^{K}_0(M,J)$ (resp. ${\rm Isom}^{K}_0(M,g)$) the corresponding closed connected Lie groups. We then have:

\begin{thm}\label{thm:Calabi} 
Suppose $(M, g, J)$ is a compact $(K, q, a)$-extremal K\"ahler manifold. Then the group ${\rm Isom}^{K}_0(M,g)$ is a maximal compact connected subgroup of ${\rm Aut}^{K}_0(M,J)$. Furthermore, if ${\rm S}_{(K,q,a)}(g)=const$, then ${\rm Aut}^{K}_0(M,J)$ is a reductive complex Lie group.  
\end{thm}

This basic result yields that each compact $(K, q, a)$-extremal K\"ahler manifold $(M, J, g)$ is invariant under a {\it maximal} torus $\T$  in the connected component of the identity of the reduced automorphism group ${\rm Aut}_{\rm red}(M,J)$, with $K \in {\rm Lie}(\T)$, thus linking to the point of view of \cite{AM}. In particular, we can deduce from Theorem~\ref{thm:Calabi} and \cite[Theorem 3]{AM} that if $(M, J)$ is toric, i.e. $\T$ is $m$-dimensional, then the $(K,q, a)$-extremal K\"ahler metrics are unique up to isometries in their K\"ahler classes (see Corollary~\ref{unique} below). Concerning the existence of conformally K\"ahler, Einstein--Maxwell metrics, Theorem \ref{thm:Calabi} and \cite[Theorem 5]{AM} yield together a complete classification of the latter on the toric complex surfaces $\mathbb{{CP}}^1 \times \mathbb{{CP}}^1$ and the Hirzebruch surfaces $\mathbb{F}_n = \mathbb{P}(\cO \oplus \cO(n))\to \mathbb{{CP}}^1$ in terms of {\it explicit} constructions given by either the Calabi Ansatz~\cite{LeB0, LeB1,KTF} or by the hyperbolic ambitoric ansatz~\cite{ambitoric1} (a riemannian analogue of the Plebanski-Damianski  explicit solutions~\cite{PD}). In practice, however, the algorithm of \cite[Theorem~5]{AM} allowing one to decide whether or not for a given K\"ahler class, a quasi-periodic holomorphic vector field $K$ and a constant $a>0$ there exists a compatible conformally-K\"ahler, Einstein--Maxwell metric is of considerable complexity, see \cite{FO}. The case $\mathbb{{CP}}^1 \times \mathbb{{CP}}^1$ has been successfully resolved by \cite{LeB0,AM} (see  also \cite{FO}):
\begin{cor} \label{c:LeBrun} Any conformally-K\"ahler, Einstein--Maxwell metric on ${\mathbb {CP}}^1 \times {\mathbb {CP}}^1,$ must be toric, and if it is not a product of Fubini-Study metrics on each factor, it must be homothetically isometric to one of the metrics constructed in \cite{LeB0}.
\end{cor} 
We also notice that similarily to the K\"ahler--Einstein and cscK cases~\cite{M,L}, Theorem~\ref{thm:Calabi} places an obstruction in terms of ${\rm Aut}_{0}^K(M,J)$ for $(M,J)$ to admit a K\"ahler metric of constant $(K, q, a)$-scalar curvature, in particular a conformally K\"ahler, Einstein--Maxwell metric. 
\begin{cor}\label{c:Hirzebruch} Let $(M, J) = {\mathbb P}(\cO \oplus \cO(1)_E) \to \mathbb{F}_n$  where $E= (\cO \oplus \cO(n)) \to {\mathbb {CP}}^1$ and $\mathbb{F}_n= {\mathbb P}(E)$ is the $n$-th Hirzebruch complex surface. Denote by $K$ the generator of the $S^1$-action on $M,$ corresponding to diagonal multiplications on the $\cO_E (1)$-factor. Then $(M, J)$ admits no K\"ahler metric of constant $(bK, q, a)$-scalar curvature for any values of $b$ and $q$.
\end{cor}

\smallskip

We now describe our second result, which is a suitable modification of the stability of the existence of extremal K\"ahler metrics under deformation of the K\"ahler class, proved by LeBrun--Simanca in \cite{LS}, see also \cite{FS}.
In our extended context, and without loss of generality by using Theorem~\ref{thm:Calabi} above, we fix a maximal real torus $\T \subset {\rm Aut}_{\rm red}(M, J)$, a real weight $q$, and study the existence of a $\T$-invariant $(K, q, a)$-extremal K\"ahler metric as a function of the  K\"ahler class $\Omega\in H^{2}_{\rm dR}(M,\R)$, the vector field $K\in{\rm Lie}(\T)$, and the real constant $a>0$. We prove the following:
\begin{thm}\label{thm:LeBrun-Simanca} Suppose (without loss of generality by \Cref{thm:Calabi}) that $\T\subset{\rm Aut}_{\rm red}(M,J)$ is a maximal real torus and $(M,J)$ admits a $\T$-invariant $(K, q, a)$-extremal K\"ahler metric $(g, \omega)$. Then, for any $g$-harmonic, $\T-$invariant, $(1,1)$-form $\alpha$, and any $H\in {\rm Lie}(\T)$, there exist $\varepsilon >0$, such that for any real numbers $|s|<\varepsilon, |t|<\varepsilon$ and $|u|<\varepsilon$, there exists a $(K+u H, q, a+s)$-extremal K\"ahler metric in the K\"ahler class $[\omega + t\alpha]$.
\end{thm}
This result provides an efficient way to obtain many new examples from known ones.

\section*{Acknowledgement} 
I would like to thank my thesis supervisor Vestislav Apostolov for his invaluable advice and for sharing his insights with me. I'm also grateful to Professors A. Futaki and H. Ono who kindly inform me that they have obtained independently a proof of \Cref{thm:Calabi} in \cite{FO1}.

\section{A familly of variational problems in K{\"a}hler geometry}
\label{s1}
In this section we recall the Apostolov--Maschler \cite{AM} moment map interpretation of the  $(K,q,a)$-scalar curvature. In \cite{AM}, the case $q=-2m+1$ is considered, but their argument works for any weight $q$ (see \cite{ACGL}).\\

\smallskip

Let $(M,J,\omega)$ be a compact K{\"a}hler manifold of real dimension $2m\geq 4$. We denote by $\mathfrak{h}_{\rm red}$ the Lie algebra of the reduced automorphism group ${\rm Aut}_{\rm red}(M,J)$, given by the real holomorphic vector field with zeros (see \cite{Gauduchon-book}). Let $K\in\mathfrak{h}_{\rm red}$ be a quasi-periodic Killing vector field generating a torus $G\subset {\rm Aut}_{\rm red}(M,J)$. It is well known that $G$ acts in a isometric hamiltonian way on $(M,\omega)$. Let $f_{(K,\omega,a)}\in C^{\infty}(M,\mathbb{R})$ be the normalized positive Killing potential of $K$, defined by the condition $\int_M f_{(K,\omega,a)}v_\omega=a$.\\      

\smallskip

We denote by $\mathcal{K}^{G}(M,\omega)$ the space of all $\omega$-compatible, $G$-invariant K\"ahler structures on $(M,\omega)$, and consider the natural action of the infinite dimensional group ${\rm Ham}^{G}(M,\omega)$, of $G$-equivariant Hamiltonian transformations of $(M,\omega)$. We have the identification $${\rm Lie}\left({\rm Ham}^{G}(M,\omega)\right)\cong C^{\infty}_{0}(M,\mathbb{R})^{G}$$ where $C^{\infty}_{0}(M,\mathbb{R})^{G}$ denote the space of smooth $G$-invariant functions with zero mean value with respect to $f_{(K,\omega,a)}^{q-2}v_\omega$, ($v_\omega=\frac{\omega^{m}}{m!}$ being the Riemanian volum form) endowed with the Poisson bracket.\\

\smallskip
For any $q\in\mathbb{R}$, the space $\mathcal{K}^{G}(M,\omega)$ carries a $q$-weighted formal K{\"a}hler structure $(\bold{J},\bold{\Omega}^{(K,q,a)})$ given by (\cite{donaldson, fujiki, AM})
\begin{eqnarray*}
\bold{\Omega}^{(K,q,a)}_J(\dot{J}_1,\dot{J}_2)&=&\frac{1}{2}\int_M {\rm Tr}(J\dot{J}_1\dot{J}_2)f^{q}_{(K,\omega,a)}v_\omega,\\
\bold{J}_J(\dot{J})&=&J\dot{J},
\end{eqnarray*}
where the tangent space of $\mathcal{K}^{G}(M,\omega)$ at $J$ is identified with the space of smooth $G$-invariant sections $\dot{J}$ of ${\rm End}(TM)$ satisfying
\begin{equation*}
\dot{J}J+J\dot{J}=0,\quad \omega(\dot{J}.,.)+\omega(.,\dot{J}.)=0.
\end{equation*}
In what follows we denote by $g_J:=\omega(.,J.)$ the K{\"a}hler metric corresponding to $J\in\mathcal{K}^{G}(M,\omega)$, and index all objects calculated with respect to $J$ similarly. On $C^{\infty}_{0}(M,\mathbb{R})^{G}$, we consider the scalar product given by,
\begin{equation*}
\langle\phi,\psi\rangle_{(K,q,a)}=\int_M\phi\psi f^{q-2}_{(K,\omega,a)}v_\omega.
\end{equation*}
\begin{thm}\label{moment-map}\cite{AM}
The action of ${\rm Ham}^{G}(M,\omega)$ on $\left(\mathcal{K}^{G}(M,\omega),\bold{J},\bold{\Omega}^{(K,q,a)}\right)$ is Hamiltonian with a momentum map given by the $\langle.,.\rangle_{(K,q,a)}$-dual of the $(K,q,a)-$scalar curvature given by \eqref{(K,nu,a)-scalar-curvature}.
\end{thm}
\begin{rem}\leavevmode
\begin{enumerate}
\item The weight $q=-2m+1$ corresponds to the conformally K\"ahler, Einstein--Maxwell case studied in \cite{AM}, and ${\rm S}_{(K,q,a)}$ computes the scalar curvature of the hermitian metric $\tilde{g}_J:=f_{(K,\omega,a)}^{-2}g_J$. 
\item If $q=0$, ${\rm S}_{(K,q,a)}(J)$ computes the so-called {\it conformal scalar curvature} $\tilde{\kappa}_J$ of the hermitian metric $\tilde{g}_J$ given by (see e.g. \cite{Gaud}),
\begin{equation*}
\tilde{\kappa}_J=(2m-1)\left\langle \tilde{W}\left(\tilde {F}_J\right),\tilde {F}_J\right\rangle_{\tilde{g}_J},
\end{equation*}  
where $\tilde {F}_J=\tilde g_J(J.,.)$ is its fundamental 2-form of $(\tilde{g}_J,J)$ and $\tilde W$ is the corresponding Weil tensor. 
\item The weight $q=-m-1$ appears in the study of Levi-K\"ahler quotients (see e.g. \cite{ACGL}).
\item For a real number $p$, one can define,
\begingroup\reqnos
\begin{equation}\label{S-nu-mu}
{\rm S}_{(K,p,q,a)}(J):=f_{(K,\omega,a)}^{p-2}{\rm S}_{(K,q,a)}(g_J).
\end{equation}\endgroup
Then the $\langle\cdot,\cdot\rangle_{(K,q-p+2,a)}$-dual of \eqref{S-nu-mu} is a momentum map for the action of ${\rm Ham}^{G}(M,\omega)$ on $\left(\mathcal{K}^{G}(M,\omega),\bold{J},\bold{\Omega}^{(K,q,a)}\right)$. Taking $(p,q)=(2,q)$ we obtain Theorem~\ref{moment-map}, wheres the value $(p,q)=(\frac{2}{m},-1)$ corresponds to the Lejmi--Upmeier moment map given by the {\it hermitian scalar curvature} of $\tilde{g}_J$ (see \cite{LU}).
\end{enumerate}
\end{rem}
\section{The extended Calabi problem}
\subsection{The $(K,q,a)$-constant scalar curvature K{\"a}hler metrics}
Following \cite{AM} we now fix the complex manifold $(M,J)$ and vary the K\"ahler form $\omega$ within a fixed K{\"a}hler class $\Omega\in H^{2}(M,\mathbb{R})\cap H^{1,1}(M,\mathbb{C})$. We also fix the compact torus $G\subset {\rm Aut}_{\rm red}(M,J)$ generated by a quasi periodic vector field $K\in {\rm Lie}(G)$, and denote by $\mathcal{K}_{\Omega}^{G}(M,J)$ the space of $G$-invariant K{\"a}hler forms $\omega\in \Omega$. Let $f_{(K,\omega,a)}$ be the normaliezed Killing potential of $K$ with respect to $\omega$, with normalization constant $a>0$, such that $f_{(K,\omega,a)}>0$. As shown in \cite[Lemma 1] {AM} we have $f_{(K,\omega^{\prime},a)}>0$ on $M$ for all $\omega^{\prime}\in\mathcal{K}_{\Omega}^{G}(M,J)$.\\

\smallskip

The space $\mathcal{K}_{\Omega}^{G}(M,J)$ is a Frechet manifold given near $\omega\in\mathcal{K}_{\Omega}^{G}(M,J)$ by the open subset of elements $\phi\in C^{\infty}(M,\mathbb{R})^{G}\slash \mathbb{R}$ such that $\omega+dd^{c}\phi>0$ is  positive definite. The tangent space of $\mathcal{K}_{\Omega}^{G}(M,J)$ at $\omega$ is identified with $C^{\infty}_{0}(M,\mathbb{R})^{G}$, the space of $G$-invariant smooth functions with mean value $0$ with respect to $f_{(K,\omega,a)}^{q-2}v_\omega$.\\

\smallskip

We then consider the following generalized Calabi problem on $\mathcal{K}^{G}_\Omega(M,J)$ (see \cite{AM}):
\begin{problem}
For a weight $q\in\R$, a quasi-periodic vector field $K$ generating a torus $G$ in ${\rm Aut}_{\rm red}(M,J)$, $\omega\in\mathcal{K}^{G}_\Omega(M,J)$ and $a>0$ such that $f_{(K,\omega,a)}>0$, does there exist $\phi\in C_{0}^{\infty}(M,\R)^{G}$ such that $\omega+dd^{c}\phi$ is $(K,q,a)$-extremal?
\end{problem}

In what follows we calculate the first variation of the $(K,q,a)$-scalar curvature along $\omega\in\mathcal{K}^{G}_\Omega(M,J)$. We denote by $D$ the Levi-Civita connection and by $\delta=D^{\star}$ the co-differential of $(M,\omega,g)$.
For a $1$-form $\alpha$ on $M$, let $D^{\pm}\alpha$ be the $J$-invariant (resp. $J$-anti-invariant) part of $D\alpha$, i.e.
\begin{equation*}
\left(D^{\pm}\alpha\right)_{X,Y}=\frac{1}{2}\left(\left(D\alpha\right)_{JX,JY}\pm\left(D\alpha\right)_{X,Y}\right).
\end{equation*}
There is a natural action on $p$-forms $\psi$ induced by $J$ as follows,
\begin{equation*}
(J\psi)(X_1,\cdots,X_p)=(-1)^{p}\psi(JX_1,\cdots,JX_p).
\end{equation*} 
The twisted differential and the twisted codifferential on $p$-forms are defined by,
\begin{align*}
d^{c}=&JdJ^{-1},\\
\delta^{c}=&J\delta J^{-1}.
\end{align*}
To simplify notation we omit below the index $(K,\omega,a)$ of $f_{(K,\omega,a)}$,
\begin{lem}\label{lem: Lich-dev}
For any $G$-invariant $1$-form $\alpha$ we have 
\begin{align}
\label{Lich-dev}
\begin{split}
2f^{2-q}\delta\delta\left(f^{q}D^{-}\alpha\right) =& 2 f^{2}\delta\delta\left(D^{-}\alpha\right)-2q f\left(\Delta\alpha,df\right)
\\
&+2q f \left(\Delta df,\alpha\right)+2q f(\delta d\alpha,df)
\\
&-q(q-1)(\alpha,d(df,df))+q(q-1)(df,d(\alpha,df)),
\end{split}
\end{align}
where $(\cdot,\cdot)$ stand for the inner product of tensors induced by $g$.
\end{lem}

\begin{proof}
Indeed, 
\begin{eqnarray*}
2f^{2-q}\delta\delta\left(f^{q}D^{-}\alpha\right)&=&2f^{2}\delta\delta(D^{-}\alpha)+2q f(\delta D^{-}\alpha)(JK)\\&&+2q f\delta\left((D^{-}\alpha)(JK,\cdot)\right)-2q(q-1)(D^{-}\alpha)(K,K).
\end{eqnarray*}
We consider the decomposition of the tensor $D^{-}\alpha$ in symmetric and skew-symmetric parts $\Psi$ and $\Phi$, respectively,
\begin{equation*}
D^{-}\alpha=\Psi+\Phi.
\end{equation*}
For any vector field $X$ on $M$ we have
\begin{align}
\label{adj}
\begin{split}
\delta\left(\Psi(X,.)\right) &= -(\Psi,DX^{\flat})+\left(\delta\Psi\right)(X),
\\
\delta\left(\Phi(X,.)\right)&=(\Phi,DX^{\flat})-\left(\delta\Phi\right)(X).
\end{split}
\end{align}
Using \eqref{adj} for $X=JK$ we get
\begin{align*}
\delta\left(\Psi(JK,.)\right)&=(\delta\Psi)(JK),\\
\delta\left(\Phi(JK,.)\right)&=-(\delta\Phi)(JK).
\end{align*}
Thus,
\begin{equation}\label{4}
2f^{2-q}\delta\delta\left(f^{q}D^{-}\alpha\right)=2f^{2}\delta\delta(D^{-}\alpha)+4\nu f(\delta\Psi)(JK)-2q(q-1)(D^{-}\alpha)(K,K).
\end{equation}
Using \cite[Lemma 1.23.4]{Gauduchon-book} and $2\Phi=d\alpha-Jd\alpha$ we have
\begin{align}\label{5}
\begin{split}
(\delta\Psi)(JK)&=-(\delta D^{-}\alpha,df)+(\delta\Phi)(df^{\sharp})\\
&=-\dfrac{1}{2}(\Delta\alpha,df)+{\rm Ric}(\grad_g f,\alpha^{\sharp})+\frac{1}{2}(\delta d\alpha,df)-\frac{1}{2}(\delta Jd\alpha,df)\\
&=-\dfrac{1}{2}(\Delta\alpha,df)+(\Delta df,\alpha)-\left(\delta D^{+}df,\alpha\right)+\frac{1}{2}(\delta d\alpha,df)\\
&=-\dfrac{1}{2}(\Delta\alpha,df)+\dfrac{1}{2}(\Delta df,\alpha)+\frac{1}{2}(\delta d\alpha,df)
\end{split}
\end{align}
where we have used the identity $(\delta Jd\alpha,df)=-(\delta^{c}d\alpha)(K)=\mathcal{L}_K\delta^{c}\alpha=0$ which holds since $K$ is Killing. Furthermore,
\begin{align}\label{6}
\begin{split}
2(D^{-}\alpha)(K,K)&=\left(D_{K}\alpha\right)(K)-\left(D_{JK}\alpha\right)(JK)\\
&=-(df,d(\alpha,df))+(d^{c}f,d(\alpha,d^{c}f))-2\alpha\left(D_{K}K\right)\\
&=-(df,d(\alpha,df))+(d^{c}f,d(\alpha,d^{c}f))+(\alpha,d(df,df))\\
&=-(df,d(\alpha,df))+(\alpha,d(df,df)),
\end{split}
\end{align}
since $(d^{c}f,d(\alpha,d^{c}f))=\mathcal{L}_K(\alpha,d^{c}f)=0$ by the $G$-invariance of $\alpha$. The result follows by substituting \eqref{5} and \eqref{6} in \eqref{4}. This completes the proof.
\end{proof}
\begin{defn}
We define the $(K,q,a)$-Lichnerowicz operator 
\begin{equation*}
\mathbb{L}_{(K,q,a)}^{g}:C^{\infty}(M,\mathbb{R})^{G}\to C^{\infty}(M,\mathbb{R})^{G},
\end{equation*}
 with respect to a metric $g$ in $\mathcal{K}^{G}_\Omega(M,J)$ by
\begin{equation*}
\mathbb{L}^{g}_{(K,q,a)}(\phi)=f^{2-q}_{(K,\omega,a)}\delta\delta\left(f^{q}_{(K,\omega,a)}D^{-}(d\phi)\right),
\end{equation*}
for $\phi\in C^{\infty}(M,\mathbb{R})^{G}$.
\end{defn}

\begin{prop}
For any variation $\dot{\omega}=dd^{c}\dot{\phi}$ of $\omega$ in $\mathcal{K}_{\Omega}^{G}(M,J)$, the first variation of the $(K,q,a)$--scalar curvature is given by 
\begin{equation}\label{var-Scal(K,nu,a)}
{\boldsymbol \delta}{\rm S}_{(K,q,a)}(\dot{\phi})=-2\mathbb{L}^{g}_{(K,q,a)}(\dot{\phi})+\left(d{\rm S}_{(K,q,a)}(\omega),d\dot{\phi}\right).
\end{equation}
\end{prop}
\begin{proof}
For a variation $\dot{\omega}=dd^{c}\dot{\phi}$ in $\mathcal{K}^{G}_{\Omega}(M,J)$, the corresponding variations of $f_{(K,\omega,a)}$, $\Delta_\omega$, ${\rm Scal}_\omega$ are given by (see e.g. ~\cite{Gauduchon-book, besse}): 
\begin{align}
\label{variations}
\begin{split}
\dot{v}_\omega=&-(\Delta_\omega\dot{\phi}) v_\omega\\
\dot{f}=&(df,d\dot{\phi})\\
\dot{\Delta}_{\omega}=&(dd^{c}\dot{\phi},dd^{c}\cdot)\\
\dot{{\rm Scal}}_\omega=&-2\mathbb{L}^{g}(\dot{\phi})+(d\,{\rm Scal}(\omega),d\dot{\phi}),
\end{split}
\end{align}
where $\mathbb{L}^{g}(\dot{\phi})=\delta\delta(D^{-}d\dot{\phi})$ is the usual Lichnerowicz operator. Then the first variation of the $(K,q,a)-$scalar curvature is given by:
\begin{eqnarray*}
{\boldsymbol \delta}{\rm S}_{(K,q,a)}(\dot{\phi})&=&-2f^{2}\mathbb{L}^{g}(\dot{\phi})+ f^{2}(d{\rm Scal}(\omega),d\dot{\phi})+{\rm Scal}(\omega)(d f^{2},d\dot{\phi})+2q f\Delta_\omega(df,d\dot{\phi})\\&&+2q(df,d\dot{\phi})\Delta_\omega f+2q f(dd^{c}f,dd^{c}\dot{\phi})-q(q-1)(df,d(df,d\dot{\phi})).
\end{eqnarray*}
By \eqref{adj} and the $G$-invariance of $\phi$ we have
\begin{equation*}
(dd^{c}\dot{\phi},dd^c f)=-\Delta(df,d\dot{\phi})+(d\Delta\dot{\phi},df).
\end{equation*}
Thus,
\begin{align}
\begin{split}\label{8}
{\boldsymbol \delta}{\rm S}_{(K,q,a)}(\dot{\phi})=& -2f^{2}\mathbb{L}^{g}(\dot{\phi})+f^{2}(d{\rm Scal}(\omega),d\dot{\phi})\\&+{\rm Scal}_\omega(d f^{2},d\dot{\phi})+2q f(df,d\Delta\dot{\phi})\\&+2q(df,d\dot{\phi})\Delta_\omega f-q(q-1)(df,d(df,d\dot{\phi})).
\end{split}
\end{align}
On the other hand we have
\begin{align}
\begin{split}\label{9}
(d {\rm S}_{(K,q,a)}(\omega),d\dot{\phi})=&f^{2}(d {\rm Scal}(\omega),d\dot{\phi})+(df^{2},d\dot{\phi}){\rm Scal}(\omega)\\&+2q(\Delta f)(df,d\dot{\phi})+2q f (d\Delta f,d\dot{\phi})\\&-q(q-1)(d\dot{\phi},d(df,df)).
\end{split}
\end{align}
By taking the difference \eqref{8}-\eqref{9} we get exactly \eqref{Lich-dev} for $\alpha=d\dot{\phi}$, which, in turn, is equal to $-2\mathbb{L}_{(K,q,a)}^{g}(\dot{\phi})$.
\end{proof}

Let $\mathfrak{h}^{K}_{\rm red}$ denote the centralizer of $K$ in $\mathfrak{h}_{\rm red}$ (i.e. the space of vector fields $H\in\mathfrak{h}_{\rm red}$ such that $[H,K]=0$) and ${\rm Aut}_{\rm red}^{K}(M,J)$ the closed connected Lie subgroup of ${\rm Aut}_{\rm red}(M,J)$ with Lie algebra $\mathfrak{h}_{\rm red}^{K}$.\\ 

\smallskip

Let $\omega\in\mathcal{K}_{\Omega}^{G}(M,J)$. To each element $\phi\in C^{\infty}_{0}(M,\mathbb{R})^{G}$ we associate a vector field $\hat{\phi}$ on $\mathcal{K}^{G}_\Omega(M,J)$, equal to $\phi$ at any point of $\mathcal{K}_{\Omega}^{G}(M,J)$. We then have $[\hat{\phi},\hat{\psi}]=0$ for any $\phi,\psi\in C^{\infty}_{0}(M,\mathbb{R})^{G}$. We will consider the natural action of ${\rm Aut}_{\rm red}^{K}(M,J)$ on $\mathcal{K}_{\Omega}^{G}(M,J)$ defined by:
\begin{equation*}
\gamma\cdot\omega=\gamma^{\star}\omega.
\end{equation*}
Consider the $1$-form $\sigma$ on $\mathcal{K}_{\Omega}^{G}(M,J)$ given by: 
\begin{equation*}
\sigma_\omega(\hat{\phi})=\int_M {\rm S}_{(K,q,a)}(\omega) \phi f_{(K,\omega,a)}^{q-2}v_\omega
\end{equation*}
where $\omega\in\mathcal{K}^{G}_\Omega(M,J)$ and $\phi\in C^{\infty}_{0}(M,\mathbb{R})^{G}$.
\begin{prop}\label{prop: d-sigma=0}
The $1$-form $\sigma$ is ${\rm Aut}_{{\rm red}}^{K}(M,J)-$invariant and we have the following expression for its first variation, 
\begin{align}
\begin{split}\label{var-sigma}
{\boldsymbol \delta}\left(\sigma(\hat{\phi})\right)_\omega(\hat{\psi})=&-2\int_M(D^-d\phi,D^-d\psi) f^{q}_{(K,\omega,a)}v_\omega\\&-\int_M{\rm S}_{(K,q,a)}(\omega)(d\psi,d\phi) f^{q-2}_{(K,\omega,a)}v_\omega.
\end{split}
\end{align}
In particular $\sigma$ is closed.
\end{prop}

\begin{proof}
Since ${\rm Aut}_{\rm red}^{K}(M,J)$ preserves the complex structure $J$ and $K$, the invariance of $\sigma$ under the action of ${\rm Aut}_{\rm red}^{K}(M,J)$ follows. Now we will calculate the first variation of the functional $\omega\mapsto\sigma_{\omega}(\phi)$. By \eqref{variations} we have for each $\psi\in C^{\infty}_{0}(M,\mathbb{R})^{G}$, 
\begin{align*}
{\boldsymbol \delta}\left(\sigma(\hat{\phi})\right)_\omega(\hat{\psi})=&\int_M\boldsymbol{\delta}{\rm S}_{(K,q,a)}(\dot{\omega})\phi f^{q-2}v_\omega+\int_M{\rm S}_{(K,q,a)}(\omega)\phi (df^{q-2},d\psi)v_\omega\\&-\int_M{\rm S}_{(K,q,a)}(\omega)\phi f^{q-2}(\Delta_\omega\psi)v_\omega\\
=&\int_M\boldsymbol{\delta}{\rm S}_{(K,q,a)}(\dot{\omega})\phi f^{q-2}v_\omega-\int_M(d{\rm S}_{(K,q,a)}(\omega),d\psi)\phi f^{q-2}v_\omega\\&-\int_M{\rm S}_{(K,q,a)}(\omega)(d\psi,d\phi) f^{q-2}v_\omega.
\end{align*}
From \eqref{var-Scal(K,nu,a)} and the above formula we readily get \eqref{var-sigma}. Thus,
\begin{equation*}
\left({\bold d}\sigma\right)_\omega\left(\hat{\phi},\hat{\psi}\right)={\boldsymbol \delta}\left(\sigma(\hat{\psi})\right)_\omega(\hat{\phi})-{\boldsymbol \delta}\left(\sigma(\hat{\phi})\right)_\omega(\hat{\psi})-\sigma_\omega([\hat{\phi},\hat{\psi}])=0.
\end{equation*} 
i.e. $\sigma$ is a closed $1-$form. 
\end{proof}

\begin{rem}
One can alternatively elaborate along the lines of \cite{AM}. For $\omega\in\mathcal{K}^{G}_\Omega(M,J)$ and $J\in\mathcal{K}^{G}(M,\omega)$ fixed, we consider the path of K\"ahler metrics $\omega_t=\omega+dd^{c}\phi_t$ with $\phi_t\in C^{\infty}_{0}(M,\mathbb{R})^{G},$ $\phi_0=0$ and $\dot{\phi}_t=\phi$. Using the equivariant Moser Lemma (see \cite[Lemma 1]{AM}) there exists a family of $G$-equivariant diffeomorphisms $\Phi_t\in {\rm Diff}^{G}_0(M)$ such that $\Phi_0=id_M$ and $\Phi_t\cdot\omega=\omega_t$. Then we have a path $J_t=\Phi_t\cdot J$ in $\mathcal{K}^{G}(M,\omega)$. Note that if $g_t=\omega_t(.,J_t.)$ then $f_{(K,\omega_t,a)}=f_{(K,\omega,a)}\circ\Phi_t$. We have,
\begin{eqnarray*}
{\boldsymbol \delta}\left(\sigma(\hat \psi)\right)_{\omega}(\hat\phi)&=&\left.\dfrac{d}{dt}\right|_{t=0}\int_{M}{\rm S}_{(K,q,a)}(\omega_{t})\psi f_{(K,\omega_{t},a)}^{q-2}v_{\omega_{t}}\\
&=&\left.\dfrac{d}{dt}\right|_{t=0}\int_{M}\Phi_{t}^{\star}\left({\rm S}_{(K,q,a)}(J_{t})\left(\psi\circ\Phi^{-1}_{t}\right)f^{q-2}_{(K,\omega,a)} v_{\omega}\right)\\
&=&\left.\dfrac{d}{dt}\right|_{t=0}\int_{M}{\rm S}_{(K,q,a)}(J_{t})\left(\psi\circ\Phi^{-1}_{t}\right)f_{(K,\omega,a)}^{q-2} v_{\omega}\\
&=&\int_{M}\left(\dot{J},DJd\psi\right)f^{q}_{(K,\omega,a)}v_{\omega}+\int_{M}{\rm S}_{(K,q,a)}(\omega)d\psi(Z)f^{q-2}_{(K,\omega,a)}v_{\omega},
\end{eqnarray*}
where we used \cite[Eq(9)]{AM} and that $Z$, $\dot{J}$ are given by: 
\begin{align*}
Z&=\left.\dfrac{d}{dt}\right|_{t=0}\Phi^{-1}_{t}\circ\Phi=-\grad_g\phi,\\
\dot{J}&=\left.\dfrac{d}{dt}\right|_{t=0}\Phi_{t}.J=-\mathcal{L}_{Z}J.
\end{align*} 
It follows that,
\begin{align*}
d\psi(Z)&=-(d\psi,d \phi),\\
\left(\dot{J},DJd\psi\right)&=-2\left(D^{-}d\phi,D^{-}d\psi\right).
\end{align*}
We thus get,
\begin{equation*}
{\boldsymbol \delta}\left(\sigma(\hat \psi)\right)_{\omega}(\hat\phi)=-2\int_{M}\left(D^{-}d\phi,D^{-}d\psi\right)f_{(K,\omega,a)}^{q}v_{\omega}-\int_{M}{\rm S}_{(K,q,a)}(\omega)(d\psi,d\phi)f^{q-2}_{(K,\omega,a)}v_{\omega}.
\end{equation*}
\end{rem}
As shown in \cite{AM}, and as it easily follows from \eqref{var-sigma}, the following expression: \begin{equation*}
c_{(\Omega,K,q,a)}:=\dfrac{\int_M {\rm S}_{(K,q,a)}(\omega)f_{(K,\omega,a)}^{q-2}v_\omega}{\int_M f_{(K,\omega,a)}^{q-2}v_\omega}.
\end{equation*} 
is a topological constant (i.e. independent of the choice of $\omega$ in the K\"ahler class $\Omega$). We consider the following $1$-form, on $\mathcal{K}^{G}_\Omega(M,J)$ given by:
\begin{equation*}
\tilde{\sigma}_\omega(\hat \phi):=\int_M\left({\rm S}_{(K,q,a)}(\omega)-c_{(\Omega,K,q,a)}\right)\phi f_{(K,\omega,a)}^{q-2}v_\omega.
\end{equation*}
\begin{lem}
The $1$-form $\tilde{\sigma}$ is closed.
\end{lem}
\begin{proof}
We consider the $1$-form $\theta$ defined on $\mathcal{K}^{G}_\Omega(M,J)$ by 
\begin{equation*}
\theta_\omega(\hat{\phi})=\int_M\phi f_{(K,\omega,a)}^{q-2}v_\omega.
\end{equation*}
We have using \eqref{variations},
\begin{align*}
{\boldsymbol \delta}(\theta_\omega(\hat{\phi}))(\hat{\psi})=&\int_M\phi (df^{q-2},d\psi)v_\omega-\int_M\phi f^{q-2}(\Delta_\omega\psi)v_\omega\\
=&-\int_M(d\psi,d\phi) f^{q-2}v_\omega.
\end{align*}
Thus,
\begin{equation*}
\left({\bold d}\theta\right)_\omega\left(\hat{\phi},\hat{\psi}\right)={\boldsymbol \delta}\left(\theta(\hat{\psi})\right)_\omega(\hat{\phi})-{\boldsymbol \delta}\left(\theta(\hat{\phi})\right)_\omega(\hat{\psi})-\theta_\omega([\hat{\phi},\hat{\psi}])=0,
\end{equation*} 
i.e. $\theta$ is closed. By \Cref{prop: d-sigma=0}, $\tilde{\sigma}=\sigma-c_{(\Omega,K,q,a)}\cdot\theta$ is closed.
\end{proof}
Since $\mathcal{K}_{\Omega}^{G}(M,J)$ is contractible, $\tilde{\sigma}$ is an exact form, so it admits a primitive functional.
\begin{defn}
We define the $(\Omega,K,q,a)$-Mabuchi energy 
\begin{equation*}
\mathcal{M}_{(\Omega,K,q,a)}:\mathcal{K}_{\Omega}^{G}(M,J)\to \mathbb{R}
\end{equation*}
as minus the primitive of the one form $\tilde{\sigma}$, i.e. 
\begin{equation*}
\tilde{\sigma}=-\bold{d}\mathcal{M}_{(\Omega,K,q,a)},
\end{equation*}
normalized by $\mathcal{M}_{(\Omega,K,q,a)}(\omega_0)=0$ for some base point $\omega_0\in\mathcal{K}^{G}_\Omega(M,J)$.
\end{defn}
\begin{rem}
By its very definition, the K\"ahler metrics in $\mathcal{K}^{G}_\Omega(M,J)$ of constant $(\Omega,K,q,a)$-scalar curvature are critical points of the $(\Omega,K,q,a)-$Mabuchi functional.
\end{rem}
\subsection{The $(\Omega,K,q,a)$-Futaki invariant}
For $\omega\in\mathcal{K}_{\Omega}^{G}(M,J)$ and $H\in\mathfrak{h}_{{\rm red}}^{K}$, we denote by  $h^{\prime}_{(H,\omega)}+\sqrt{-1}h_{(H,\omega)}\in C^{\infty}_0(M,\mathbb{C})$ the normalized holomorphy potantial of $H$, i.e. $h^{\prime}_{(H,\omega)}$ and $h_{(H,\omega)}$ are the normalized smooth functions such that, 
\begin{equation*}
H=\grad_{g}(h^{\prime}_{(H,\omega)})+J\grad_g (h_{(H,\omega)}).
\end{equation*}
Using the identification $T_\omega\mathcal{K}_{\Omega}^{G}(M,J)\cong C^{\infty}_{0}(M,\mathbb{R})^{G}$, the vector field $JH$ defines a vector field $\widehat{JH}$ on $\mathcal{K}_{\Omega}^{G}(M,J)$, given by:
\begin{equation*}
\omega\mapsto\mathcal{L}_{JH}\omega=dd^{c}h_{(H,\omega)},
\end{equation*}
so that $\widehat{JH}_\omega=h_{(H,\omega)}$. By the invariance of $\tilde{\sigma}$ under the ${\rm Aut}_{\rm red}^{K}(M,J)$-action and Cartan's formula we get,
\begin{equation*}
\mathcal{L}_{\widehat{JH}}\tilde{\sigma}=\bold{d}\left(\tilde{\sigma}_\omega(\widehat{JH})\right)=0.
\end{equation*}
Then $\omega\mapsto\tilde{\sigma}(\widehat{JH})$ is constant on $\mathcal{K}_{\Omega}^{G}(M,J)$. We will use the following notation, 
\begin{equation*}
\mathcal{F}_{(\Omega,K,q,a)}(H):=\tilde{\sigma}(\widehat{JH})=\int_M\left({\rm S}_{(K,q,a)}(\omega)-c_{(\Omega,K,q,a)}\right)h_{(H,\omega)} f_{(K,\omega,a)}^{q-2}v_\omega.
\end{equation*}
\begin{defn}\label{Def-Fut}
 The linear map $\mathcal{F}_{(\Omega,K,q,a)}:\mathfrak{h}_{red}^{K}\to \mathbb{R}$ will be called the $(\Omega,K,q,a)$-Futaki invariant associated to the data $(\Omega,K,q,a)$.
\end{defn}

\begin{rem}\leavevmode
\begin{enumerate}
\item \Cref{Def-Fut} is consistant with the one given in \cite{AM} for $q=-2m+1$, but it has the advantage to show that the $(\Omega,K,q,a)-$Futaki invariant extends to the whole of $\mathfrak{h}_{\rm red}^{K}$, not just ${\rm Lie}(G)$.

\item For a K\"ahler class $\Omega$ which admits $(K,q,a)$-extremal metric $\omega$, $\mathcal{F}_{\Omega,K,q,a}=0$ if and only if  $\omega$ is a $(K,q,a)$-constant scalar curvature K\"ahler metric. In fact, for $\Xi=J \grad_g\left({\rm S}_{(K,q,a)}(\omega)\right)\in\mathfrak{h}^{K}_{\rm red}$ we have
\begin{equation*}
\mathcal{F}_{(\Omega,K,q,a)}\left(\Xi\right)=\int_M\left({\rm S}_{(K,q,a)}(\omega)-c_{(\Omega,K,q,a)}\right)^{2}f^{q-2}_{(K,\omega,a)}v_\omega=0,
\end{equation*}
thus ${\rm S}_{(K,q,a)}(\omega)=c_{(\Omega,K,q,a)}$.
\end{enumerate}
\end{rem}

\section{Proof of \Cref{thm:Calabi}}\label{s4} 
In this section we shall prove \Cref{thm:Calabi} from the introduction. We thus assume that $(g,\omega)$ is an $(K,q,a)$-extremal metric on a compact, connected K\"ahler manifold $(M,J)$, and $G\subset {\rm Aut}_{\rm red}(M,J)$ the torus generated by the quasi-periodic Killing vector field $K$. 
\begin{lem}
For any $G$-invariant function $\phi\in C^{\infty}(M,\mathbb{R})^{G}$ we have
\begin{equation*}
\mathcal{L}_\Xi\phi=-2f^{2-q}_{(K,\omega,a)}\delta\delta\left(f_{(K,\omega,a)}^{q}D^{-}(d^{c}\phi)\right),
\end{equation*}
where $\mathcal{L}_\Xi$ denotes the Lie derivative along the vector field $\Xi=J\grad\left({\rm S}_{(K,q,a)}(\omega)\right)$.
\end{lem}
\begin{proof}
We have,
\begin{eqnarray*}
\mathcal{L}_\Xi\phi&=&-(d{\rm S}_{(K,q,a)}(\omega),d^{c}\phi)\\
&=&-f^{2}(d {\rm Scal}(\omega),d^{c}\phi)-2q f(d\Delta f,d^{c}\phi)+q(q-1)(d^{c}\phi,d(df,df)).
\end{eqnarray*} 
By taking $\alpha=d^{c}\phi$ in \eqref{Lich-dev} we get,
\begin{eqnarray*}
2f^{2-q}(D^{-}d)^{\star}f^{q}(D^{-}d^{c})\phi&=&2f^{2}(D^{-}d)^{\star}(D^{-}d^{c})\phi+2q f(d\Delta f,d^{c}\phi)\\&&-q(q-1)(d^{c}\phi,d(df,df))\\
&=&f^{2}(d {\rm S}_\omega,d^{c}\phi)+2q f(d\Delta f,d^{c}\phi)\\&&-q(q-1)(d^{c}\phi,d(df,df)),
\end{eqnarray*}
where we have used (see \cite[p.63, Eq.(1.23.15)]{Gauduchon-book}), 
\begin{equation*}
(D^{-}d)^{\star}(D^{-}d^{c})\phi=(d {\rm S}_\omega,d^{c}\phi),
\end{equation*}
and the identity, 
\begin{equation*}
(\delta dd^{c}\phi,df)=-(\delta^{c}d^{c}d\phi)(K)=-\mathcal{L}_K\delta^{c}d^{c}\phi=0.
\end{equation*}
\end{proof}
For a $1$-forme $\alpha$ we denote by $D^{2,0}\alpha$ (resp. $D^{0,2}\alpha$) the $(2,0)$-part (resp. $(0,2)$-part) of the tensor $D\alpha$. We define the $(K,q,a)$-Calabi's operators $\mathbb{L}^{g,\pm}_{(K,q,a)}$ on $C^{\infty}(M,\mathbb{C})^{G}$ by 
\begin{eqnarray*}
\mathbb{L}_{(K,q,a)}^{g,+}(F)&=&2f_{(K,\omega,a)}^{2-q}(D^{0,2}d)^{\star}f_{(K,\omega,a)}^{q}D^{0,2}dF\\
\mathbb{L}_{(K,q,a)}^{g,-}(F)&=&2f^{2-q}_{(K,\omega,a)}(D^{2,0}d)^{\star}f^{q}_{(K,\omega,a)}D^{2,0}dF.
\end{eqnarray*}
Recall that the space of hamiltonian Killing vector fields is given by (see \cite{Gauduchon-book}) 
\begin{equation*}
\mathfrak{k}_{\rm ham}=\mathfrak{h}_{\rm red}\cap \mathfrak{k}.
\end{equation*}
The following Proposition is straifhtforword (see \cite[Chapter 2]{Gauduchon-book}).
\begin{prop}\label{prop-L}\leavevmode
\begin{enumerate}
\item\label{(i)-L} Let $H=\grad_g P+J\grad_g Q$, where $P,Q$ are real valued functions with zero mean, such that $[H,K]=0$. Then $H\in\mathfrak{h}_{\rm red}$ if and only if $\mathbb{L}_{(K,q,a)}^{g,+}(P+\sqrt{-1}Q)=0$ and $P,Q\in C^{\infty}(M,\mathbb{R})^{G}$ i.e. we have
\begin{equation*}
\mathfrak{h}_{\rm red}^{K}\cong\ker(\mathbb{L}_{(K,q,a)}^{g,+})\cap C^{\infty}_{0}(M,\mathbb{C})^{G}.
\end{equation*}
\item\label{(ii)-L} For any $F\in C^{\infty}(M,\mathbb{C})^{G}$ we have, 
\begin{equation*}
\mathbb{L}_{(K\nu,a)}^{g,\pm}(F)=\mathbb{L}^{g}_{(K,q,a)}(F)\pm\frac{\sqrt{-1}}{2}\mathcal{L}_{\Xi}F.
\end{equation*}
\item\label{(iii)-L} Let $X$ be real holomorphic vector field. Then $X\in\mathfrak{k}_{\rm ham}^{K}$ if and only if there exists $h\in C^{\infty}(M,\mathbb{R})^{G}$ such that $X=J \grad_g h$ and $\mathbb{L}_{(K,q,a)}^{g}(h)=0$.  
\end{enumerate}
\end{prop}

\begin{thm}\label{thm:calabi} Suppose $(M, g, J)$ is a compact $(K, q, a)$-extremal K\"ahler manifold.  Then $\mathfrak{h}^K$ admits the following $\langle\cdot, \cdot\rangle_{(g,q,a)}$-orthogonal decomposition
\begin{equation}\label{h-split}
\mathfrak{h}^{K}=\mathfrak{h}^{K}_{(0)}\oplus\left(\bigoplus_{\lambda>0}\mathfrak{h}^{K}_{(\lambda)}\right),
\end{equation} 
where $\mathfrak{h}^{K}_{(0)}$ is the centralizer of $\Xi$ in $\mathfrak{h}^{K}$ and for $\lambda>0$, $\mathfrak{h}^{K}_{(\lambda)}$ denote the subspace of elements $X\in\mathfrak{h}^{K}$ such that $\mathcal{L}_{\Xi}X=\lambda JX$.\\
The subspace $\mathfrak{h}^K_{(0)}$ is a reductive complex Lie subalgebra of $\mathfrak{h}^{K}$; it contains $\mathfrak{a}$, $\mathfrak{k}_{\rm ham}^{K}$ and $J\mathfrak{k}_{\rm ham}^{K}$ and is given by the $\langle\cdot, \cdot\rangle_{(g,q,a)}$-orthogonal sum of these three spaces:
\begin{equation}\label{h0-split}
\mathfrak{h}^K_{(0)}=\mathfrak{a}\oplus\mathfrak{k}_{\rm ham}^{K}\oplus J\mathfrak{k}_{\rm ham}^{K}.
\end{equation}
Moreover, each $\mathfrak{h}^{K}_{(\lambda)}$, $\lambda>0$, is contained in the ideal $\mathfrak{h}_{\rm red}^{K}$, so that we also have the following $\langle\cdot, \cdot\rangle_{(g,q,a)}$-orthogonal decompositions:
\begin{align}\label{h-k-split}
\begin{split}
\mathfrak{h}^{K}=&\mathfrak{a}\oplus \mathfrak{h}_{\rm red}^{K},\\
\mathfrak{k}^{K}=&\mathfrak{a}\oplus \mathfrak{k}_{\rm ham}^{K}.
\end{split}
\end{align}
\end{thm}

\begin{proof}
Let $X=X_H+\grad_g P+J\grad_g Q\in\mathfrak{h}^{K}$, where $X_H$ is the dual of the harmonic part of $\xi=X^{\flat}$ denoted $\xi_H$, and $P,Q\in C^{\infty}(M,\mathbb{R})$ with zero mean value. Since $K=J \grad_g f$ is Killing we have 
\begin{equation*}
\mathcal{L}_K P=\mathcal{L}_KQ=0.
\end{equation*} 
By \eqref{Lich-dev} in Lemma \ref{lem: Lich-dev} we have
\begin{eqnarray*}
2f^{2-q}(D^{-}d)^{\star}f^{q}D^{-}\xi_H&=&2f^{2}(D^{-}d)^{\star}D^{-}\xi_H\\
&=&f^{2}(d{\rm Scal}(\omega),\xi_H)+2q f(d\Delta f,\xi_H)\\&&-q(q-1)(\xi_H,d(df,df))\\
&=&J\mathcal{L}_{\Xi}\xi_H=0,
\end{eqnarray*}
where we have used $(\xi_H,df)=0$ and the fact that $\Xi$ is a Killing vector field. It follows that 
\begin{equation*}
0=f^{2-q}(D^{-}d)^{\star}f^{q}D^{-}\xi=f^{2-q}(D^{-}d)^{\star}f^{q}D^{-}(dP+d^{c}Q)={\rm Re}\left(\mathbb{L}_{(K,q,a)}(P+\sqrt{-1}Q)\right).
\end{equation*}
Starting from $JX$ instead of $X$ we similarly get  
\begin{equation*}
{\rm Im}\left(\mathbb{L}_{(K,q,a)}^{+}(P+\sqrt{-1}Q)\right)=0.
\end{equation*}
It follows that $\mathbb{L}_{(K,q,a)}^{+}(P+\sqrt{-1}Q)=0$, then by \Cref{prop-L}\ref{(i)-L} we have that $X_H$ and $\grad_gP+J\grad_g Q$ are real holomorphic vector fields, which proves \eqref{h-k-split} (for the decomposition of $\mathfrak{k}^{K}$  we use the fact that $\mathfrak{k}_{\rm ham}:=\mathfrak{k}\cap\mathfrak{h}_{\rm red}$ and $\mathfrak{k}\cap \mathfrak{a}=\mathfrak{a}$).\\
Since $\Xi$ is Killing and commutes with $K$, the operators $\mathbb{L}_{(K,q,a)}^{g,\pm}$ commute. Then $\mathbb{L}_{(K,q,a)}^{g,-}$ acts on $\mathfrak{h}_{red}^{K}$ and by \Cref{prop-L}\ref{(ii)-L} this action is given by $-\sqrt{-1}\mathcal{L}_{\Xi}$. Since $\mathbb{L}^{-}_{(K,q,a)}$ is $\langle\cdot,\cdot\rangle_{(g,q,a)}$-self-adjoint and semi-positive, $\mathfrak{h}_{\rm red}^{K}$ splits as 
\begin{equation*}
\mathfrak{h}_{\rm red}^{K}=\mathfrak{h}^{K}_{{\rm red},(0)}\oplus\left(\bigoplus_{\lambda>0}\mathfrak{h}^{K}_{(\lambda)}\right),
\end{equation*}
where $\mathfrak{h}^{K}_{{\rm red},(0)}$ is the kernel of $\mathcal{L}_{\Xi}$ in $\mathfrak{h}_{\rm red}^{K}$ whereas, for each $\lambda>0$, $\mathfrak{h}^{K}_{(\lambda)}$ is the subspace of elements $X\in\mathfrak{h}^{K}$ such that $\mathcal{L}_{\Xi}X=\lambda JX$. Using \eqref{h-k-split} we get \eqref{h-split}(Notice that $\mathfrak{h}^{K}_{(\lambda)}=\mathfrak{h}^{K}_{{\rm red},(\lambda)}$ since $\Xi$ is Killing and commutes with $K$).\\
We have $\mathfrak{a}\oplus\mathfrak{k}_{\rm ham}^{K}\oplus J\mathfrak{k}_{\rm ham}^{K}\subset \mathfrak{h}^K_{(0)}$. By \Cref{prop-L}\ref{(ii)-L} the restriction of $\mathcal{L}_{\Xi}$ to $\ker\left(\mathbb{L}_{(K,q,a)}^{g,+}\right)\cap C^{\infty}_{0}(M,\mathbb{C})^{G}$ coincides with the restriction of $\mathbb{L}_{(K,q,a)}^{g}$ to the same space. Then, using \Cref{prop-L} \ref{(iii)-L}, we obtain the converse inclusion, which proves  \eqref{h0-split}.
\end{proof}
Now we are in position to give a proof for \Cref{thm:Calabi}.
\begin{proof}[\bf Proof of \Cref{thm:Calabi}]
This is done as in the case where $(G=\{1\},q=0,a=0)$ (see \cite{Gauduchon-book,calabi}). Let $\mathfrak{s}$ be the Lie algebra of a connected, compact Lie subgroup, $S\subset {\rm Aut}^{K}_0(M,J)$ containing ${\rm Isom}^{K}_0(M,g)$. Suppose, for a contradiction, that there exists $X\in\mathfrak{s}$ that doesn't belong to $\mathfrak{k}^{K}$. By \Cref{thm:calabi}, (see \eqref{h-split}, \eqref{h-k-split} and \eqref{h0-split}) we have the splitting
\begin{equation*}
\mathfrak{h}^K=\mathfrak{k}^{K}\oplus J\mathfrak{k}_{\rm ham}^{K}\oplus\left(\bigoplus_{\lambda>0}\mathfrak{h}_{(\lambda)}^{K}\right),
\end{equation*}
then we can assume that $X\in J\mathfrak{k}_{\rm ham}^{K}\oplus\left(\bigoplus_{\lambda>0}\mathfrak{h}_{(\lambda)}^{K}\right)$. Let $X=X_0+\sum_{\lambda>0}X_\lambda$ be the corresponding decomposition of $X$, then for any positive integer $r$ we have 
\begin{equation*}
\left(\mathcal{L}_\Xi\right)^{2r}X=-\sum_{\lambda>0}\lambda^{2r}X_\lambda \in \mathfrak{s}.
\end{equation*}
It follows that each component $X_\lambda$ of $X$ is in $\mathfrak{s}$. We can therefore assume that $X\in\mathfrak{s}_{\lambda}:=\mathfrak{s}\cap \mathfrak{h}^{K}_{(\lambda)}$ or $X\in J\mathfrak{k}_{\rm ham}^{K}\subset \mathfrak{s} _0$. Suppose that $X\in \mathfrak{s}_{\lambda}$ for some $\lambda>0$. Let $B$ denote the Killing form of $\mathfrak{s}$. Since $S$ is a compact Lie group, $B$ is semi-negative and it's kernel coincides with the center of $\mathfrak{s}$. On the other hand $X$ belongs to the kernel of $B$, indeed for any $Y\in \mathfrak{s}_{\lambda_1}$ and $Z\in \mathfrak{s}_{\lambda_2}$, by Jacobi identity  we can easily show that $[X,[Y,Z]]\in\mathfrak{s}_{\lambda+\lambda_1+\lambda_2}\neq \mathfrak{s}_{\lambda_2}$ then $\mathfrak{s}_{\lambda+\lambda_1+\lambda_2}=\{0\}$ and by consequence $[X,[Y,Z]]=0$. It follows that for any $Y\in \mathfrak{s}$ we have $B(X,Y)=0$. Hence $X$ belongs to the center of $\mathfrak{s}$, but we have $\Xi\in \mathfrak{k}^{K}\subset \mathfrak{s}$ and $[X,\Xi]=-\lambda JX\neq 0$, a contradiction.\\
It follows that $X\in J\mathfrak{k}_{\rm ham}^{K}$. Then $X=\grad_g(P)$ for some real function $P$. By the hypothesis the flow $\Phi^{X}_t$ of $X$ is contained in a compact connected subgroup of ${\rm Aut}^{K}_0(M,J)$. It follows that $X$ is quasi-periodic with a flow closure in ${\rm Aut}^{K}_0(M,J)$ given by a torus $T^{k}$ of dimension $k\geq1$. Note that $k\neq1$ since a gradient vector field does not admit any non-trivial closed integral curve, as $\frac{d}{dt}P\left(\Phi^{X}_t(x)\right)=|X|^{2}_{\Phi^{X}_t(x)}\geq 0$. It follows that $k>1$. Let $x\in M$ such that $X_x\neq 0$. We have that $P(\Phi^{X}_t(x))$ is an increasing function of $t$, so that $P(\Phi^{X}_t(x))-P(x)>c$, for $t>1$, where $c>0$. But by density of $\Phi^{X}_t$ in the torus $T^{k}$, $\Phi^{X}_t$ meets any small neighborhood $U$ of $x$, which is a contradiction. We conclude that $\mathfrak{s}=\mathfrak{k}^{K}$. \\

\smallskip

If the $(K,q,a)$-scalar curvature is constant then by \Cref{thm:calabi}, $\mathfrak{h}^{K}$ splits as 
\begin{equation*}
\mathfrak{h}^{K}=\mathfrak{a}\oplus\mathfrak{k}_{\rm ham}^{K}\oplus J\mathfrak{k}_{\rm ham}^{K},
\end{equation*}
since $\mathfrak{h}_{(\lambda)}^{K}=\{0\}$. In particular $\mathfrak{h}^{K}$ is a reductive complex Lie algebra.
\end{proof}
We have the following immediate consequences of \Cref{thm:Calabi}. 
\begin{cor}\label{cor:T-inv}
Any $(K,q,a)$-extremal metric on a compact K\"ahler manifold $(M,J)$ belongs to $\mathcal{K}^{\T}_\Omega(M,J)$ for some maximal torus $\T$ of ${\rm Aut}_{\rm red}(M,J)$ such that $K\in{\rm Lie}(\T)$.
\end{cor}

\begin{cor}\label{unique}
Let $g$ and $\tilde{g}$ be two $(K,q,a)-$extremal metrics on $(M,J)$. Then there is $\Phi\in {\rm Aut}^{K}_0(M,J)$ such that ${\rm Isom}^{K}_0(M,g)={\rm Isom}^{K}_0(M,\Phi^{\star}\tilde{g})$. Furthermore if $(M,J)$ is a toric manifold and $g$ and $\tilde{g}$ are two $(K,q,a)$-extremal metrics in the same K\"ahler class $\Omega$, then they are isometric (see \cite{AM}).
\end{cor}

\begin{proof}[\bf Proof of \Cref{c:LeBrun}]
This follows from \Cref{cor:T-inv} and \cite[Proposition 6]{AM}
\end{proof}

\begin{proof}[\bf Proof of \Cref{c:Hirzebruch}]
We have the following exact sequence (see \cite[Proposition 1.3]{ACGT3}):
\begin{equation*}
0\rightarrow \mathfrak{h}_B(M)\rightarrow \mathfrak{h}(M)\rightarrow \mathfrak{h}(B)\rightarrow 0
\end{equation*}
where $B=\mathbb{F}_n$ and $\mathfrak{h}_B(M)$ denote the Lie algebra of holomorphic vector fields on $M$ which are tangent to the fibers of $\pi$. The proof of \cite[Proposition 1.3]{ACGT3} also shows that, 
\begin{equation*}
0\rightarrow \mathfrak{h}^{K}_B(M)\rightarrow \mathfrak{h}^{K}(M)\rightarrow \mathfrak{h}(B)\rightarrow 0
\end{equation*}
where $\mathfrak{h}^{K}_B(M)={\rm span}_{\C}\{K,JK\}$ is the abelian sub-algebra generated by the vector fields $K$, $JK$. If $M$ admits a K\"ahler metric of constant $(bK, q, a)$-scalar curvature, then $\mathfrak{h}^{K}(M)$ must be reductive by \Cref{thm:Calabi}. As $\mathfrak{h}^{K}_B(M)$ is in the center of $\mathfrak{h}^{K}(M)$, it would follow that $\mathfrak{h}(B)$ is reductive, which is not the case for $B=\mathbb{F}_n$ (see e.g. \cite{besse}). It follows that $M$ admits no K\"ahler metric of constant $(bK, q, a)$-scalar curvature.
\end{proof}

\section{The $(K,q,a)$-extremal K{\"a}hler metrics relatively to a maximal torus $\mathbb{T}$, and $(\mathbb{T},K,q,a)$-extremal vector field}\label{s5}
Using Corollaries \ref{cor:T-inv} and \ref{unique}, we assume from now on that $\mathbb{T}$ is a fixed  maximal torus in ${\rm Aut}_{\rm red}(M,J)$ and $K\in{\rm Lie}(\T)$. We denote by $\Pi^{\mathbb{T}}_g$ the orthogonal projection with respect to the $L^{2}-$scalar product 
\begin{equation*}
\langle\phi,\psi\rangle_{(g,q,a)}:=\int_M\phi\psi f_{(K,\omega,a)}^{q-2}v_\omega
\end{equation*} 
defined on the Hilbert space $L^{2}_\mathbb{T}(M,\mathbb{R})$ onto the space $P^{\mathbb{T}}_g(M,\R)$ of Killing potentials of the elements of ${\rm Lie}(\T)$ relatively to $g$ which is isomorphic to $\mathbb{R}\oplus {\rm Lie}(\T)$. Then we have the following decomposition of the $(K,q,a)$-scalar curvature,
\begin{equation*}
{\rm S}_{(K,q,a)}(\omega)={\rm S}_{(K,q,a)}^{\T}(\omega)+\Pi^{\mathbb{T}}_g \left({\rm S}_{(K,q,a)}(\omega)\right),
\end{equation*}
  
\begin{defn}
We call ${\rm S}_{(K,q,a)}^{\mathbb{T}}(\omega)$ the reduced $(K,q,a)$-scalar curvature with respect to $\mathbb{T}$. We say that $\omega\in\mathcal{K}_{\Omega}^\mathbb{T}(M,J)$ is $(K,q,a)$-extremal relatively $\mathbb{T}$ if ${\rm S}_{(K,q,a)}^{\mathbb{T}}(\omega)$ is identically zero.
\end{defn} 

\begin{rem}
Notice that by \Cref{cor:T-inv}, any $(K,q,a)$-extremal metric is extremal relatively to the maximal torus of ${\rm Aut}_{\rm red}(M,J)$ containing $K$.   
\end{rem}

Following \cite[Proposition 4.11.1]{Gauduchon-book} we have,
\begin{defn}
For $X,Y\in\mathfrak{h}_{\rm red}$ with normalized complex potentials $F^{X}_\omega,F^{Y}_\omega$ we define the $(\Omega,K,q,a)$-Futaki-Mabuchi bilinear by the following expression
\begin{equation*}
\mathcal{B}_{(\Omega,K,q,a)}(X,Y):=\int_M F^{X}_\omega F^{Y}_\omega f_{(K,\omega,a)}^{q-2}v_\omega
\end{equation*}
which is independent from the choice of $\omega\in\mathcal{K}^{\mathbb{T}}_\Omega(M,J)$.
\end{defn}
We denote by $Z_{\omega}^{\mathbb{T}}(K,q,a)$ the vector field given by
\begin{equation*}
Z_{\omega}^{\mathbb{T}}(K,q,a):=J\grad_g\left(\Pi^{\mathbb{T}}_g\left({\rm S}_{(K,q,a)}(\omega)\right)\right).
\end{equation*}   
Then for all $H\in {\rm Lie}(\T)$, we have
\begin{equation}\label{F=B}
\mathcal{F}_{(\Omega,K,q,a)}(H)=-\mathcal{B}_{(\Omega,K,q,a)}\left(H,Z_{\omega}^{\T}(K,q,a)\right).
\end{equation}
From its very definition, the restriction of $\mathcal{B}_{(\Omega,K,q,a)}$ to ${\rm Lie}(\T)$ is negative definite. Then $Z_{\omega}^{\mathbb{T}}(K,q,a)$ is well-defined by the above expression, so it is an element of ${\rm Lie}(\mathbb{T})$, independent of the choice of $\omega\in\mathcal{K}^{\mathbb{T}}_\Omega(M,J)$.  
\begin{defn}
We call $Z^{\mathbb{T}}(\Omega,K,q,a)\in {\rm Lie}(\T)$ the $(\Omega,K,q,a)$-extremal vector field.
\end{defn}
Now we consider the $1$-form $\zeta^{\mathbb{T}}$ defined on $\mathcal{K}^{\mathbb{T}}_{\Omega}(M,J)$ by,
\begin{equation*}
\zeta^{\mathbb{T}}_\omega(\hat{\phi})=\int_M \Pi_{g}^{\mathbb{T}}\left({\rm S}_{(K,q,a)}(\omega)\right)\phi f^{q-2}_{(K,\omega,a)}v_\omega.
\end{equation*}
\begin{lem}\label{d-zeta}
The $1$-form $\zeta^{\mathbb{T}}$ is closed.
\end{lem}
\begin{proof}
To simplify notations we denote $z(\omega):=\Pi_{g}^{\mathbb{T}}\left({\rm S}_{(K,q,a)}(\omega)\right)$.  Then  $Z^{\mathbb{T}}(\Omega,K,q,a)=J \grad_g z(\omega)$. For a variation $\dot{\omega}=dd^{c}\phi$ in $\mathcal{K}^{\mathbb{T}}_{\Omega}(M,J)$, using \eqref{variations} we have
\begin{equation*}
\dot{z}(\phi)=(d\phi,dz(\omega))_{\omega}
\end{equation*}
and therefore, 
\begin{eqnarray*}
{\boldsymbol \delta}\left(\zeta^{\mathbb{T}}(\hat{\psi})\right)_{\omega}(\hat{\phi})&=&\int_M (d\phi,dz(\omega))_{\omega}\psi f^{q-2}v_\omega+\int_M z(\omega)\psi(d\phi,df^{q-2})v_{\omega}\\&&-\int_M z(\omega)\psi\Delta_\omega\phi f^{q-2}v_\omega\\
&=&-\int_M z(\omega)(d\phi,d\psi)f^{q-2}v_{\omega}.
\end{eqnarray*}
It follows that,
\begin{equation*}
\left(\bold{d}\zeta^{\mathbb{T}}\right)_\omega(\hat{\phi},\hat{\psi})={\boldsymbol \delta}\left(\zeta^{\mathbb{T}}(\hat{\psi})\right)_{\omega}(\hat{\phi})-{\boldsymbol \delta}\left(\zeta^{\mathbb{T}}(\hat{\phi})\right)_{\omega}(\hat{\psi})=0.
\end{equation*}
\end{proof}
Now we consider the $1$-form $\sigma^{\mathbb{T}}$ on $\mathcal{K}^{\mathbb{T}}_{\Omega}(M,J)$ given by
\begin{equation*}
\sigma^{\mathbb{T}}:=\sigma -\zeta^\mathbb{T}
\end{equation*}
which is a closed $1$-form by virtue of \Cref{prop: d-sigma=0} and \Cref{d-zeta}.
\begin{defn}
The relative Mabuchi energy $\mathcal{M}^{\mathbb{T}}_{(\Omega,K,q,a)}$ is defined by 
\begin{equation*}
\sigma^{\mathbb{T}}=-\bold{d}\mathcal{M}^{\mathbb{T}}_{(\Omega,K,q,a)},
\end{equation*}
where the primitive $\mathcal{M}^{\mathbb{T}}_{(\Omega,K,q,a)}$ is normalized by requiring $\mathcal{M}^{\mathbb{T}}_{(\Omega,K,q,a)}(\omega_0)=0$ for some base point $\omega_0\in\mathcal{K}^{\mathbb{T}}_{\Omega}(M,J)$.
\end{defn}
\begin{rem} 
By its very definition, the critical points of the relative Mabuchi energy are the $\mathbb{T}$--invariant $(K,q,a)-$extremal metrics.
\end{rem}

\section{Proof of \Cref{thm:LeBrun-Simanca}}\label{s6}
Let $(M,J)$ be a compact K\"ahler manifold. We fix $K\in\mathfrak{h}_{\rm red}$, and $q\in\mathbb{R}$. Suppose that $(g,\omega)$ is a $(K,q,a)$-extremal K\"ahler metric on $M$ with $\Omega=[\omega]$. Without loss of generality, by \Cref{cor:T-inv}, we can assume that $(g,\omega)$ is invariant under the action of a maximal torus $\T\sub {\rm Aut}_{\rm red}(M,J)$. Let $\alpha$ be $\T-$invariant $g$-harmonic $(1,1)$-form. We take $(\omega,\alpha)=0$ to avoid trivial deformations of the form $\alpha=\lambda\omega$. We denote by
\begin{equation*}
\omega_{t,\phi}:=\omega+t\alpha+dd^{c}\phi,
\end{equation*}
a $\mathbb{T}-$invariant deformations of $\omega$ for $t\in\mathbb{R}$ and $\phi\in C^{\infty}(M,\mathbb{R})^{\mathbb{T}}$. We consider the following map, 
\begin{equation*}
{\rm S}:\mathbb{R}^{3}\times C^{\infty}(M,\mathbb{R})^{\mathbb{T}}\rightarrow C^{\infty}(M,\mathbb{R})^{\mathbb{T}}
\end{equation*}
defined by,
\begin{equation*}
{\rm S}(s,t,u,\phi):={\rm S}_{(K+u H,q,a+s)}(\omega_{t,\phi}),
\end{equation*}
so that ${\rm S}(0)={\rm S}_{(K,q,a)}(\omega):={\rm S}$. We denote $f_{(s,t,u,\phi)}=f_{(K+u H,\omega_{t,\phi},a+s)}>0$ the hamiltonian function of $K+u H$ with respect to $\omega_{t,\phi}$, with normalization constant $a+s$, so that $f_{0}=f_{(K,\omega,a)}:=f$. We take $k>n$ such that the Sobolev space $L^{2}_{k}(M,\R)^{\mathbb{T}}$ form an algebra for the usual multiplication of functions, embadded in $C^{4}(M,\mathbb{R})^{\mathbb{T}}$. Then ${\rm S}$ defines a map
\begin{equation*}
{\rm S}:\mathbb{R}^{3}\times L^{2}_{k+4}(M,\R)^{\mathbb{T}}\rightarrow L^{2}_{k}(M,\R)^{\mathbb{T}},
\end{equation*}
and we have:
\begin{lem}\label{DS}
The map ${\rm S}$ is $C^{1}$ with Fr{\'e}chet derivative in $0$ given by
\begin{equation*}
{\rm D}_{0}{\rm S}=\begin{pmatrix} 
A & B & C & D
\end{pmatrix}
\end{equation*}
with,
\begin{align*}
A=&2f {\rm Scal}(\omega)+2q\Delta(f),\\
B=&-2\left(\rho_{(K,q,a)}(\omega),\alpha\right)+2\lambda f {\rm S}+2q\Delta (f)+2q f\Delta\lambda-2q(q-1)(d\lambda,df),\\
C=&2 f_{(H,\omega)}f{\rm Scal}(\omega)-2q(q-1)(K,H)+2q\left[f_{(H,\omega)}\Delta(f)+f\Delta(f_{(H,\omega)})\right],\\
D(\dot{\phi})=&-2\mathbb{L}_{(K,q,a)}^{g}(\phi)+\left(d{\rm S},d\dot{\phi}\right),
\end{align*}
where $(.,.)$, $\grad$, $\Delta$, the green operator $\mathbb{G}$ are calculated with respect to $\omega$, and $\lambda:=-\mathbb{G}(\alpha,dd^{c}f)$.
\end{lem}
\begin{proof}
The expressions of $A$, $C$ and $D$ are straightforward. For the partial derivative with respect to $t$ we have ${\rm S}_{(K,q,a)}(\omega)=2\Lambda_\omega\left(\rho_{(K,q,a)}(\omega)\right)$
where (see \cite{ACGL})
\begin{equation*}
\rho_{(K,q,a)}(\omega)=f^{2}\rho(\omega)-q f dd^{c}f-\frac{1}{2}q(q-1)df\wedge d^{c}f,
\end{equation*}
with $\rho(\omega)$ is the Ricci form of $(g,\omega)$. By taking $X=-JK$ in \cite[Lemma 5.2.4]{Gauduchon-book} we get,
\begin{equation*}
\left.\frac{\partial}{\partial t}\right|_{0}f_{(s,t,u,\phi)}=-\mathbb{G}\left(\delta(\alpha(K,))\right)=-\mathbb{G}\left(\alpha,dd^{c}f_{(K,\omega,a)}\right)=\lambda,
\end{equation*} 
and we have $\left.\frac{\partial}{\partial t}\right|_{0}\rho(\omega_{t,\phi})=0$ since we assumed $(\omega,\alpha)=0$. Thus,
\begin{eqnarray*}
B&=&2\left(\left.\frac{\partial}{\partial t}\right|_{0}\Lambda_{\omega_{t,\phi}}\right)\rho_{(K,q,a)}(\omega)+2\Lambda_\omega\left(\left.\frac{\partial}{\partial t}\right|_{0}\rho_{(K,q,a+s)}(\omega_{t,\phi})\right)\\
&=&-2\left(\rho_{(K,q,a)}(\omega),\alpha\right)+2\lambda f {\rm S}+2q\Delta f+2q f\Delta\lambda-2q(q-1)(d\lambda,df).
\end{eqnarray*}
\end{proof}
We consider the following maps,
\begin{eqnarray*}
\mathcal{F}(s,t,u)&:=&\mathcal{F}_{([\omega+t\alpha],K+uH,q,a+s)},\\
\mathcal{B}(s,t,u)&:=&\mathcal{B}_{([\omega+t\alpha],K+uH,q,a+s)},\\
Z(s,t,u)&:=&Z^{\T}([\omega+t\alpha],K+uH,q,a+s).
\end{eqnarray*}
\begin{lem}\label{d-F-B-t}
The $t$-derivative of the character $\mathcal{F}(s,t,u)$ and the bilinear for $\mathcal{B}(s,t,u)$ in the point $(s,t,u)=0$ is given by
\begin{align}\label{d-F-t}
\begin{split}
\left.\frac{\partial}{\partial t}\right|_{0}\mathcal{F}(s,t,u)(X)=&\left\langle h_{(X,\omega)},B\right\rangle_{(g,q,a)}+\left\langle h_{(X,\omega)},(q-2)\lambda f{\rm S}\right\rangle_{(g,q,a)}\\&-\left\langle h_{(X,\omega)},f^{2-q}.\left(\alpha,dd^{c}\mathbb{G}\left(f^{q-2}{\rm S}\right)\right)\right\rangle_{(g,q,a)},
\end{split}
\end{align}

\begin{align}\label{d-B-t} 
\begin{split}
\left.\frac{\partial}{\partial t}\right|_{0}\mathcal{B}(s,t,u)(X,Y)=&\left\langle\alpha,f^{2-q}h_{(X,\omega)}dd^{c}\mathbb{G}\left(h_{(Y,\omega)}f^{q-2}\right)\right\rangle_{(g,q,a)}\\
&+\left\langle\alpha,f^{2-q}\mathbb{G}\left(h_{(X,\omega)}f^{q-2}\right)dd^{c}h_{(Y,\omega)}\right\rangle_{(g,q,a)}\\&+(q-2)\left\langle\alpha,f^{2-q}\mathbb{G}\left(h_{(X,\omega)}h_{(Y,\omega)} f^{q-3}\right)dd^{c}f\right\rangle_{(g,q,a)},
\end{split}
\end{align}

for any $X=J \grad_g(h_{(X,\omega)})$ and $Y=J \grad_g (h_{(Y,\omega)})$ in ${\rm Lie}(\T)$ with $h_{X,\omega}$, $h_{Y,\omega}$ are the normalized real potential of $-JX$, $-JY$ respectively.
\end{lem}

\begin{proof}
For the derivative of $\mathcal{F}(s,t,u)$ we have, 
\begin{equation*}
\mathcal{F}(s,t,u)(X)=\int_M{\rm S}(s,t,u)h_{(X,\omega+t\alpha)}f^{q-2}_{(s,t,u)}v_{\omega+t\alpha},
\end{equation*}
then,
\begin{align*}
\left.\frac{\partial}{\partial t}\right|_{0}\mathcal{F}(s,t,u)(X)=&\int_M B h _{(X,\omega)}f^{q-2}v_\omega-(q-2)\int_M{\rm S}\lambda f^{q-3}h_{(X,\omega)}v_\omega\\&-\int_M{\rm S}f^{q-2}\mathbb{G}\left(\delta(\alpha(X,\cdot))\right)v_\omega.
\end{align*}
On the other hand,
\begin{align*}
\int_M{\rm S}f^{q-2}\mathbb{G}\left(\delta(\alpha(X,.))\right)v_\omega=&\int_M\left(\alpha(X,.),d\mathbb{G}\left(f^{q-2}{\rm S}\right)\right) v_\omega\\
=&\int_M\left(\alpha,X^{\flat}\wedge d\mathbb{G}\left(f^{q-2}{\rm S}\right)\right) v_\omega\\
=&\int_M\left(\alpha,d^{c}h\wedge d\mathbb{G}\left(f^{q-2}{\rm S}\right)\right) v_\omega\\
=&\int_M h.\left(\alpha,dd^{c}\mathbb{G}\left(f^{q-2}{\rm S}\right)\right) v_\omega\\
=&\left\langle h_{(X,\omega)},f^{2-q}\left(\alpha,dd^{c}\mathbb{G}\left(f^{q-2}{\rm S}\right)\right)\right\rangle_{(g,q,a)}.
\end{align*}
which gives the expression \eqref{d-F-t} for the $t$-derivative of $\mathcal{F}(s,t,u)$.\\

\smallskip

Now we calculate the $t$-derivative of $\mathcal{B}(s,t,u)$. We have
\begin{equation*}
\mathcal{B}(s,t,u)(X,Y)=-\int_M h_{(X,\omega+t\alpha)}h_{(Y,\omega+t\alpha)}f^{q-2}_{(s,t,u)}v_{\omega+t\alpha}.
\end{equation*}
Then
\begin{align*}
\left.\frac{\partial}{\partial t}\right|_{0}\mathcal{B}(s,t,u)(X,Y)=&\int_M \mathbb{G}\left(\delta(\alpha(X,\cdot))\right)h_{Y}f^{q-2}v_\omega\\&+\int_M \mathbb{G}\left(\delta(\alpha(Y,\cdot))\right)h_{X}f^{q-2}v_\omega\\&+(q-2)\int_M \mathbb{G}\left(\delta(\alpha(K,\cdot))\right)h_{X}h_Y f^{q-3}v_\omega.
\end{align*}
On the other hand,
\begin{eqnarray*}
\int_M \mathbb{G}\left(\delta(\alpha(X,.))\right)h_{Y}f^{q-2}v_\omega&=&\int_M \left(\alpha(X,\cdot),d\mathbb{G}\left(h_{Y}f^{q-2}\right)\right)v_\omega\\
&=&\int_M \left(\alpha,d^{c}h_X\wedge d\mathbb{G}\left(h_{Y}f^{q-2}\right)\right)v_\omega\\
&=&\left\langle\alpha,f^{2-q}h_Xdd^{c}\mathbb{G}\left(h_{Y}f^{q-2}\right)\right\rangle_{(g,q,a)},
\end{eqnarray*}

\begin{eqnarray*}
\int_M \mathbb{G}\left(\delta(\alpha(Y,.))\right)h_{X}f^{q-2}v_\omega&=&\int_M \left(\alpha(Y,\cdot),d\mathbb{G}\left(h_{X}f^{q-2}\right)\right)v_\omega\\
&=&\int_M \left(\alpha,d^{c}h_Y\wedge d\mathbb{G}\left(h_{X}f^{q-2}\right)\right)v_\omega\\
&=&\int_M \left(\alpha,\mathbb{G}\left(h_{X}f^{q-2}\right)dd^{c}h_Y\right)v_\omega\\
&=&\left\langle\alpha,f^{2-q}\mathbb{G}\left(h_{X}f^{q-2}\right)dd^{c}h_Y\right\rangle_{(g,q,a)},
\end{eqnarray*}
and, 
\begin{equation*}
\int_M \mathbb{G}\left(\delta(\alpha(K,.))\right)h_{X}h_Y f^{q-3}v_\omega=\left\langle\alpha,f^{2-q}\mathbb{G}\left(h_{X}h_Y f^{q-3}\right)dd^{c}f\right\rangle_{(g,q,a)}.
\end{equation*}
Which proves \eqref{d-B-t}. 

\end{proof}

In the following lemma we give the $s$ and $u$-derivatives of $\mathcal{F}(s,t,u)$ and $\mathcal{B}(s,t,u)$ in $(s,t,u)=(0,0,0)$. We omit the proof since it follows from straightforward calculations.

\begin{lem}\label{d-F-B-K-a}\leavevmode
\begin{enumerate}
\item The $s$-derivative of $\mathcal{F}(s,t,u)$ is given by
\begin{equation}\label{d-F-s}
\left.\frac{\partial}{\partial s}\right|_{0}\mathcal{F}(s,t,u)(X)=\left\langle q f^{-1}{\rm S}_{q-1},h_{(X,\omega)}\right\rangle_{(g,q,a)}.
\end{equation}
where ${\rm S}_{q-1}:={\rm S}_{(K,q-1,a)}(\omega)$.

\item The $u$-derivative of $\mathcal{F}(s,t,u)$

\begin{equation}\label{d-F-u}
\left.\frac{\partial}{\partial u}\right|_{0}\mathcal{F}(s,t,u)(X)=\left\langle C+(q-2)f^{-1}f_{(H,\omega)}{\rm S},h_{(X,\omega)}\right\rangle_{(g,q,a)}.
\end{equation}

\item The $s$-derivative of $\mathcal{B}(s,t,u)$ is given by
\begin{equation}\label{d-B-s}
\left.\frac{\partial}{\partial s}\right|_{0}\mathcal{B}(s,t,u)=(q-2)\mathcal{B}_{(\Omega,K,q-1,a)}.
\end{equation}

\item The $u$-derivative of $\mathcal{B}(s,t,u)$ is given by
\begin{equation}\label{d-B-u}
\left.\frac{\partial}{\partial u}\right|_{0}\mathcal{B}(s,t,u)(X,Y)=(q-2)\int_M h_{(X,\omega)}h_{(Y,\omega)}f_{(H,\omega)}f^{q-3}v_\omega
\end{equation}
for any $X=J \grad_g(h_{(X,\omega)})$ and $Y=J \grad_g (h_{(Y,\omega)})$ in ${\rm Lie}(\T)$.
\end{enumerate}
\end{lem}

\begin{lem}
Let $\omega$ be a $(K,q,a)$-extremal metric, we have
\begin{enumerate}\label{d-Z-s-t-u}
\item The $t$-derivative of $Z(s,t,u)$ is given by  
\begin{equation}\label{d-Z-t}
\left.\frac{\partial}{\partial t}\right|_{0}Z(s,t,u)=J \grad_{g}\left(\Pi^{\mathbb{T}}_{g} \left[B+\mathbb{G}\left(\alpha,dd^{c}{\rm S}\right]\right)\right).
\end{equation}
\item The $s$-derivative of $Z(s,t,u)$ is given by
\begin{equation}\label{d-Z-s}
\left.\frac{\partial}{\partial s}\right|_{0}Z(s,t,u)=J\grad_g\left(\Pi_{g}^{\mathbb{T}}\left[f^{-1}\left(q{\rm S}_{q-1}+{\rm S}\right)\right]\right).
\end{equation}

\item 
\begin{equation}\label{d-Z-K}
\left.\frac{\partial}{\partial u}\right|_{0}Z(s,t,u)=J\grad_g\left(\Pi_{g}^{\mathbb{T}}\left[C+2(q-2)f^{-1}f_{(H,\omega)}{\rm S}\right]\right).
\end{equation}
\end{enumerate}
\end{lem}

\begin{proof}\leavevmode
\begin{enumerate}
\item We have $\left.\frac{\partial}{\partial t}\right|_{0}Z(s,t,u)=J \grad_g(P_{g})$ for some function $P_{g}\in P^{\mathbb{T}}_g(M,\R)$, since $Z(s,t,u)\in {\rm Lie}(\T)$ for all $(s,t,u)$. By \eqref{F=B}, for all $X\in{\rm Lie}(\T)$ we have,
\begin{equation}\label{ss}
\mathcal{B}(s,t,u)(Z(s,t,u),X)=-\mathcal{F}(s,t,u)(X)
\end{equation}
then
\begin{eqnarray*}
\mathcal{B}(0)\left(\left.\frac{\partial}{\partial t}\right|_{0}Z(s,t,u),X\right)&=&-\langle P_{g},h_{(X,\omega)}\rangle_{(g,q,a)}\\
&=&-\left.\frac{\partial}{\partial t}\right|_{0}\mathcal{F}(s,t,u)(X)-\left(\left.\frac{\partial}{\partial t}\right|_{0}\mathcal{B}(s,t,u)\right)\left(Z(0),X\right).
\end{eqnarray*}
Using \eqref{d-F-t}, \eqref{d-B-t} and the fact that $\omega$ is $(K,q,a)-$extremal we get, 
\begin{equation*}
P_{g}=\Pi^{\T}_g \left(B+\mathbb{G}\left(\alpha,dd^{c}{\rm S}\right)\right).
\end{equation*} 
\item We have $\left.\frac{\partial}{\partial s}\right|_{0}Z(s,t,u)=J \grad_g(Q_{g})$ for some function $Q_{g}\in P^{\mathbb{T}}_g(M,\R)$, since $Z(s,t,u)\in {\rm Lie}(\T)$ for all $(s,t,u)$. Taking the derivative of \eqref{ss} with respect to $s$ we get,
\begin{align*}
-\mathcal{B}(0)\left(\left.\frac{\partial}{\partial s}\right|_{0}Z(s,t,u),X\right)=&\langle Q_{g},h_{(X,\omega)}\rangle_{(g,q,a)}\\=&\left.\frac{\partial}{\partial s}\right|_{0}\mathcal{F}(s,t,u)(X)+\left(\left.\frac{\partial}{\partial s}\right|_{0}\mathcal{B}(s,t,u)\right)\left(Z(0),X\right)\\
=&\left\langle f^{-1}\left(q{\rm S}_{(q-1}+{\rm S}\right),h_{(X,\omega)}\right\rangle_{(g,q,a)}
\end{align*}

where we used the fact that $\omega$ is $(K,q,a)$-extremal and \eqref{d-F-s}, \eqref{d-B-s}. Thus
\begin{equation*}
Q_{g}=f^{-1}\left(q{\rm S}_{(K,q-1,a)}(\omega)+{\rm S}_{(K,q,a)}(\omega)\right)
\end{equation*}
which proves the result. 
\item This is done similarly to \eqref{d-Z-t} and \eqref{d-Z-s} by using \eqref{ss}, \eqref{d-F-u} and \eqref{d-B-u}.
\end{enumerate}
\end{proof}

We denote by $\Pi^{\T}_{(s,t,u,\phi)}$ the orthogonal projection on $\mathcal{P}_{g_{t,\phi}}^{\T}$ with respect to the inner product $\langle\cdot,\cdot\rangle_{g_{t,\phi},q,a+s}$.
 
\begin{lem}\label{ext-pot}
For a $(K,q,a)$-extremal metric $\omega$ we have, 

\begin{equation}\label{ext-pot-t}
\left.\frac{\partial}{\partial t}\right|_{0}\Pi_{(s,t,u,\phi)}^{\mathbb{T}}{\rm S}(s,t,u,\phi)=\Pi_{g}^{\mathbb{T}} B+(\Pi_{g}^{\mathbb{T}}-\Id)\left(\mathbb{G}(\alpha,dd^{c}{\rm S}\right).
\end{equation}
  
\begin{equation}\label{ext-pot-s}
\left.\frac{\partial}{\partial s}\right|_{0}\Pi_{(s,t,u,\phi)}^{\mathbb{T}}{\rm S}(s,t,u,\phi)=\Pi_{g}^{\mathbb{T}}\left[f^{-1}\left(q{\rm S}_{q-1}+{\rm S}\right)\right].
\end{equation}

\begin{equation}\label{ext-pot-u}
\left.\frac{\partial}{\partial u}\right|_{0}\Pi_{(s,t,u,\phi)}^{\mathbb{T}}{\rm S}(s,t,u,\phi)=\Pi_{g}^{\mathbb{T}}\left[C+2(q-2)f^{-1}f_{(H,\omega)}{\rm S}\right].
\end{equation}
\end{lem}

\begin{proof}
We have $Z(s,t,u)=J\grad_{g_{t,\phi}}(\Pi^{\T}_{(s,t,u,\phi)}{\rm S}(s,t,u,\phi))$ then 
\begin{equation*}
J \grad_g\left(\left.\frac{\partial}{\partial t}\right|_{0}(\Pi_{(s,t,u,\phi)}^{\T}{\rm S}(s,t,u,\phi)\right)=\left.\frac{\partial}{\partial t}\right|_{0}Z(s,t,u)-J\left
(\left.\frac{\partial}{\partial t}\right|_{0}\grad_{g_{t,\phi}}\right)(\Pi_{g}^{\T}{\rm S}).
\end{equation*}
On the other hand we have,
\begin{equation*}
\left
(\left.\frac{\partial}{\partial t}\right|_{0}\grad_{g_{t,\phi}}\right)(\Pi_{g}^{\T}{\rm S})=\left(\alpha(Z(0),.)\right)^{\sharp}=\grad_g\left(\mathbb{G}\left(\alpha,dd^{c}{\rm S}\right)\right).
\end{equation*}
By \eqref{d-Z-t} it follows that
\begin{equation*}
\left.\frac{\partial}{\partial t}\right|_{0}\Pi_{(s,t,u,\phi)}^{\T}{\rm S}(s,t,u,\phi)=\Pi^{\T}_g B+(\Pi^{\T}_g-I)\left(\mathbb{G}(\alpha,dd^{c}{\rm S})\right)+c.
\end{equation*}
By differentiating in $t=0$ the equality,
\begin{equation*}
\int_M(\Pi^{\T}_{g_t}{\rm S}(\omega_t))f^{q-2}_{(K,\omega_t,a)}=\int_M{\rm S}(\omega_t)f^{q-2}_{(K,\omega_t,a)}
\end{equation*}
we get $c=0$, which proofs \eqref{ext-pot-t}. Similarly we can show \eqref{ext-pot-s} and \eqref{ext-pot-u}.
\end{proof}

Following LeBrun-Simanca's arguments \cite{LS} we give a proof of \Cref{thm:LeBrun-Simanca}.

\begin{proof}
Let $L^{2}_{k}(M,\R)^{\mathbb{T},\perp}$ be the orthogonal complement of $P^{\mathbb{T}}_g(M,\R)$ with respect to $\langle\cdot,\cdot\rangle_{(g,q,a)}$ in $L^{2}_{k}(M,\R)^{\mathbb{T}}$. For $t\in(-\epsilon,\epsilon)$ and $\phi\in U$ where $U$ is a small neighborhood of the origin in $L^{2}_{k+4}(M,\R)^{\mathbb{T}}$. As in \cite{LS} by taking a smaller open set $U$ and smaller $\epsilon$ we may assume that,
\begin{equation*}
\ker\left(\Id-\Pi^{\mathbb{T}}_{g}\right)\circ\left(\Id-\Pi^{\mathbb{T}}_{(s,t,u,\phi)}\right)=\ker\left(\Id-\Pi^{\mathbb{T}}_{(s,t,u,\phi)}\right).
\end{equation*}
Now we consider the LeBrun-Simanca map
\begin{equation*}
\Psi:(-\epsilon,\epsilon)^{3}\times U\rightarrow(-\epsilon,\epsilon)^3\times L^{2}_{k}(M,\R)^{\mathbb{T},\perp} 
\end{equation*}
defined by
\begin{equation*}
\Psi(s,t,u,\phi):=\left(s,t,\left(\Id-\Pi^{\mathbb{T}}_{g}\right)\circ\left(\Id-\Pi^{\mathbb{T}}_{(s,t,u,\phi)}\right){\rm S}(s,t,u,\phi)\right).
\end{equation*}
Note that $\Psi(0)=0$ and if $\Psi(s,t,u,\phi)=(s,t,u,0)$ then $\omega_{t,\phi}$ is $(K+uH,q,a+s)$-extremal.\\

\smallskip 

The map $\Psi$ is $C^{1}$ and its Fr{\'e}chet derivative at the origin is given by:
\begin{eqnarray*}
{\rm D}_{0}\Psi&=&\begin{pmatrix} 
1&0&0&0 \\
0&1&0&0 \\
0&0&1&0  \\
0&0&0&\Id-\Pi^{\mathbb{T}}_{g}
\end{pmatrix}\begin{pmatrix}
1&0&0&0 \\
0&1&0&0 \\
0&0&1&0  \\
A&B+\mathbb{G}(\alpha,dd^{c}{\rm S})& C&-2\mathbb{L}^{g}_{(K,q,a)}
\end{pmatrix}
\end{eqnarray*}
where $A$, $B$, and $C$ are given in \Cref{DS}. Indeed, by \Cref{ext-pot} we have,  
\begin{eqnarray*}
\left.\dfrac{\partial}{\partial \phi}\right|_{0}\left(\Id-\Pi^{\mathbb{T}}_{(s,t,u,\phi)}\right){\rm S}(s,t,u,\phi).\dot{\phi}&=&D(\dot{\phi})-\left.\dfrac{\partial}{\partial \phi}\right|_{0}\Pi^{\mathbb{T}}_{(s,t,u,\phi)}{\rm S}(s,t,u,\phi)\\
&=&D(\dot{\phi})-(d{\rm S},d\dot{\phi})\\
&=&-2f^{2-q}(D^{-}d)^{\star}f^{q}(D^{-}d)\dot{\phi}.
\end{eqnarray*}

\begin{eqnarray*}
\left.\dfrac{\partial}{\partial t}\right|_{0}\left(\Id-\Pi^{\mathbb{T}}_{(s,t,u,\phi)}\right){\rm S}(s,t,u,\phi)&=&B-\left.\dfrac{\partial}{\partial t}\right|_{0}\Pi^{\mathbb{T}}_{(s,t,u,\phi)}{\rm S}(s,t,u,\phi)\\
&=&B-\Pi_{g}^{\mathbb{T}}B+\left(\Id-\Pi_{g}^{\mathbb{T}}\right)\left(\mathbb{G}(\alpha,dd^{c}{\rm S})\right)\\
&=&\left(\Id-\Pi^{\mathbb{T}}_{(g,a)}\right)\left(B+\mathbb{G}(\alpha,dd^{c}{\rm S})\right).
\end{eqnarray*}

\begin{eqnarray*}
\left.\dfrac{\partial}{\partial s}\right|_{0}\left(\Id-\Pi^{\mathbb{T}}_{(s,t,u,\phi)}\right){\rm S}(s,t,u,\phi)&=&A-\left.\dfrac{\partial}{\partial s}\right|_{0}\Pi^{\mathbb{T}}_{(s,t,u,\phi)}{\rm S}(s,t,u,\phi)\\
&=&A-\Pi_{g}^{\mathbb{T}}\left[f^{-1}\left(q{\rm S}_{q-1}+{\rm S}\right)\right].
\end{eqnarray*}

\begin{eqnarray*}
\left.\dfrac{\partial}{\partial u}\right|_{0}\left(\Id-\Pi^{\mathbb{T}}_{(s,t,u,\phi)}\right){\rm S}(s,t,u,\phi)&=&C-\left.\dfrac{\partial}{\partial s}\right|_{0}\Pi^{\mathbb{T}}_{(s,t,u,\phi)}{\rm S}(s,t,u,\phi)\\
&=&\left(\Id-\Pi_{g}^{\mathbb{T}}\right)C-\Pi_{g}^{\mathbb{T}}\left[2(q-2)f^{-1}f_{(H,\omega)}{\rm S}\right].
\end{eqnarray*}
The operator $\mathbb{L}^{g}_{(K,q,a)}$ is a formally $\langle\cdot,\cdot\rangle_{(g,q,a)}$-self-adjoint, $\mathbb{T}-$invariant, elliptic fourth-order differential operator and extends to a continuous linear operator,
\begin{equation*}
\mathbb{L}^{g}_{(K,q,a)}:L^{2}_{k+4}(M,\R)^{\mathbb{T},\perp}\rightarrow L^{2}_{k}(M,\R)^{\mathbb{T},\perp}
\end{equation*}
which is an isomorphism (since $\T$ is a maximal torus of ${\rm Aut}_{\rm red}(M,J)$). Thus 
\begin{equation*}
D\Psi_{0}:\mathbb{R}^{3}\times L^{2}_{k+4}(M,\R)^{\mathbb{T},\perp}\rightarrow \mathbb{R}^{3}\times L^{2}_{k}(M,\R)^{\mathbb{T},\perp}
\end{equation*}
is an isomorphisme. It follows from the inverse function theorem that $\Psi$ is an isomorphisme in a neighborhood $(-\epsilon,\epsilon)^{2}\times U $ of $0$. Using the Sobolev embbeding theorem, we can assume that the solution is of regularity at least $C^{4}$. We conclude using a similar bootstraping argument as in the case of extremal metrics \cite[Proposition 4]{LS}.
\end{proof}


\begin{thebibliography}{10}



\bibitem{abreu0} M.~Abreu, {\it K{\"a}hler geometry of toric varieties and
extremal metrics}, Internat. J. Math {\bf 9} (1998), 641--651.

\bibitem{abreu} M.~Abreu, {\it K{\"a}hler metrics on toric orbifolds}, J. Differential Geom. {\bf 58} (2001), 151--187.

\bibitem{AG}  V. Apostolov and P. Gauduchon, {\it The Riemannian Goldberg--Sachs Theorem}, Internat. J. Math. {\bf 8} (1997), 421--439.

\bibitem{ambitoric1} V.~Apostolov, D.~M.~J.~Calderbank, and P.~Gauduchon, {\it Ambitoric geometry I: Einstein metrics and extremal ambik\"ahler structures}, J. Reine Angew. Math. 721 (2016), 109-147.

\bibitem{ambitoric2} V.~Apostolov, D.~M.~J.~Calderbank and P.~Gauduchon, Paul, {\it Ambitoric geometry II: Extremal  toric  surfaces and Einstein 4-orbifolds}, Ann. Sci. \'Ec. Norm. Sup\'er.  4\`eme s\'erie, 48 (2015),  1075-1112.

\bibitem{ACG} V.~Apostolov, D.~M.~J. Calderbank and P. Gauduchon, {\it Weakly self-dual K\"ahler surfaces}, Compos. Math. {\bf 135} (2003), 279--322.

\bibitem{ACGL}V.~Apostolov, D.~M.~J. Calderbank, E.~Legendre and P. Gauduchon, {\it Levi-K\"ahler reduction of CR structures, products of spheres, and toric geometry}, in preparation.

\bibitem{ACGT2} V.~Apostolov, D.~M.~J. Calderbank, P.~Gauduchon and C.~T{\o}nnesen-Friedman, {\it Hamiltonian $2$-forms in K\"ahler geometry} II {\it Global classification}, J. Differential Geom. {\bf 68} (2004), 277--345.

\bibitem{ACGT3} V.~Apostolov, D.~M.~J. Calderbank, P.~Gauduchon and C.~T{\o}nnesen-Friedman, {\it Extremal K\"ahler metrics on ruled manifolds and stability}. G\'eom\'etrie diff\'erentielle, physique math\'ematique, math\'ematiques et soci\'et\'e (II).  Ast\'erisque {\bf 322} (2008), 93-150.

\bibitem{ACGT4} V.~Apostolov, D.~M.~J. Calderbank, P.~Gauduchon and C.~T{\o}nnesen-Friedman, {\it Extremal K\"ahler metrics on projective bundles over a curve}. Adv. Math. {\bf 227} (2011),  2385--2424.


\bibitem{AM} V.~Apostolov, G.~Maschler,{\it Conformally K\"ahler, Einstein-Maxwell geometry}, arXiv:1512.06391v1, to appear in JEMS.

\bibitem{B-B} L. B\'erard-Bergery, {\it Sur de nouvelles vari\'et\'es riemanniennes d'Einstein},  Publ. de l'Institut E. Cartan (Nancy) {\bf 4} (1982), 1--60.

\bibitem{BW} R.~J.~Berman and  D.~ Witt-Nystr\"om, {\it Complex optimal transport and the pluripotential theory of K\"ahler-Ricci solitons},  arXiv:1401.8264.

\bibitem{besse} A.~L.~Besse, {\rm Einstein Manifolds}, {\it Ergebnisse} (3) {\bf 10}, Springer-Verlag, Berlin-Heidelberg-New York, 1987.

\bibitem{Bryant} R. Bryant, {\it Bochner--K{\"a}hler metrics}, J. Amer. Math. Soc. {\bf 14} (2001), 623--715.

\bibitem{calabi} E.~Calabi, {\it Extremal K{\"a}hler metrics}, Seminar on Differential Geometry, pp. 259--290, Ann. of Math. Stud. {\bf 102}, Princeton Univ. Press, Princeton, N.J., 1982.

\bibitem{CLW} X.~X.~Chen, C.~LeBrun and B.~Weber, {\it On Conformally K\"ahler, Einstein Manifolds}, J. Amer. Math. Soc. {\bf 21} (2008), 1137--1168.

\bibitem{DS} V.~Datar and  G.~Sz\'ekelyhidi, {\it K\"ahler--Einstein metrics along the smooth continuity method}, arXiv:1506.07495.

\bibitem{DG} L. David and P. Gauduchon, {\it The Bochner-flat geometry of weighted projective spaces}. in `Perspectives in Riemannian geometry', 109"1¤7156, CRM Proc. Lecture Notes, 40, Amer. Math. Soc., Providence, RI, 2006.

\bibitem{Delzant} T. Delzant, {\it Hamiltoniens p\'eriodiques et image convexe de l'application moment}, Bull. Soc. Math. France {\bf 116} (1988), 315--339.

\bibitem{De} A. Derdzi\'nski, {\it Self-dual K{\"a}hler manifolds and Einstein
manifolds of dimension four}, Compos. Math. 49 (1983), 405--433.

\bibitem{DM0} A.~Derdzinski and G.~Maschler, Special K\"ahler-Ricci
potentials on compact K\"ahler manifolds. J. reine angew. Math. 593 (2006), 73--116.

\bibitem{DM} A.~Derdzinski and G.~Maschler, {\it A moduli curve for compact
conformally-Einstein K\"ahler manifolds}, Compos. Math. {\bf 141} (2005),
1029--1080.

\bibitem{D} R.~Dervan, {\it  Uniform stability of twisted constant scalar curvature K\"ahler metrics},  arXiv:1412.0648.

\bibitem{donaldson} S.~K.~Donaldson, {\it Remarks on gauge theory, complex geometry and 4-manifold topology}, Fields Medallists"1¤7 lectures, 384--403, World Sci. Ser. 20th Century Math., 5, World Sci. Publishing, River Edge, NJ, 1997.

\bibitem{Do-02} S.~K.~Donaldson, {\it Scalar curvature and stability of toric varieties}, J. Differential Geom. {\bf 62} (2002), 289--349.

\bibitem{Donaldson2} S.~K.~Donaldson, {\it Interior estimates for solutions of Abreu's equation}, Collect. Math. {\bf 56} (2005) 103--142.

\bibitem{fujiki} A.~Fujiki, {\it Moduli space of polarized algebraic manifolds and K\"ahler metrics},
[translation of Sugaku {\bf 42}, no. 3 (1990), 231-243], Sugaku Expositions {\bf 5}, no. 2 (1992), 173--191.

\bibitem{FS} A. Fujiki, G. Schumacher, {\it The moduli space of extremal compact K\"ahler manifolds and generalized Weil--Petersson metrics}, Publ. Res. Inst. Math. Sci. {\bf 26} (1990) 101--183.

\bibitem{Futaki1} A.~Futaki {\it An obstruction to the existence of Einstein K\"ahler metrics}, Invent. Math. {\bf 73} (1983), 437--443.

\bibitem{Futaki2} A.~Futaki, {\it On compact K\"ahler manifolds of constant scalar curvature}, Proc. Japan Acad., Ser. A, {\bf 59} (1983), 401--402.

\bibitem{FO} A.~Futaki and H.~Ono, {\it Volume minimization and Conformally K\"ahler, Einstein--Maxwell geometry}, arXiv:1706.07953.

\bibitem{FO1} A.~Futaki and H.~Ono, {\it Conformally Einstein--Maxwell K\"ahler metrics and structure of the automorphism group}, arXiv:1708.01958.



\bibitem{FOW} A.~Futaki, H.~Ono and G.~Wang, {\it Transverse K\"ahler geometry of Sasaki manifolds and toric Sasaki--Einstein manifolds}, Journal of Differential Geometry, {\bf 83} (2009), 585--635.

\bibitem{FM} A.~Futaki and T.~Mabuchi, {\it Bilinear forms and extremal K\"ahler vector fields associated with K\"ahler classes}, Math. Ann. {\bf 301} (1995), 199--210.

\bibitem{Gaud} P.~Gauduchon, {\it La 1-forme de torsion d'une vari\'et\'e hermitienne compacte}, Math. Ann. {\bf 267} (1984), no. 4, 495{-}518. 

\bibitem{Gauduchon-book} P.~Gauduchon, {\rm Calabi's extremal metrics: An elementary introduction}, Lecture Notes.

\bibitem{Guan} D.~Guan, {\it On modified Mabuchi functional and Mabuchi moduli space of K\"ahler metrics on toric bundles}, Math. Res. Let. {\bf 6} (1999), 547--555.

\bibitem{guillemin} V.~Guillemin, {\it K{\"a}hler structures on toric
varieties}, J. Differential Geom. {\bf 40} (1994), 285--309.

\bibitem{GS} V.~Guillemin and S.~Sternberg, {\it Riemann sums over polytopes},  Ann. Inst. Fourier (Grenoble) 57 (2007), no. 7, 2183--2195.

\bibitem{kob-0} S.~Kobayashi, Transformation groups in differential geometry, Ergebnisse der Mathematik und ihrer Grenzgebiete, {\bf 70}, Springer-Verlag, New York-Heidelberg, 1972.

\bibitem{KTF} C.~Koca and C.~W.~T{\o}nnesen-Friedman, {\it Strongly Hermitian Einstein-Maxwell Solutions on Ruled Surfaces}, Annals of Global Analysis and Geometry, {\bf 50} (2016), 29--46.


\bibitem{KN} {S.~Kobayashi and K.~Nomizu},  {\it Foundations of Differential Geometry}, I,II, Interscience Publishers (1963).

\bibitem{LeB0} C.~LeBrun, {\it The Einstein--Maxwell equations, K\"ahler metrics, and Hermitian geometry}, J. Geom. Phys. {\bf 91} (2015), 163--171.

\bibitem{LeB1} C.~LeBrun, {\it The Einstein--Maxwell Equations and Conformally K\"ahler Geometry}, Comm. Math. Phys. {\bf 344} (2016), 621-653.

\bibitem{L}A.~Lichnerowicz, {\it G\'eom\'etrie des groupes de transformation}, Travaux et Recherches Math\'ematiques 3, Dunod (1958).

\bibitem{LS} C.R. LeBrun, S. Simanca, {\it On the K\"ahler classes of extremal metrics}, in: Geometry and Global Analysis, Sendai, 1993, Tohoku Univ., Sendai, 1993, 255--271.


\bibitem{legendre} E.~Legendre, {\it Toric geometry of convex quadrilaterals},  J. Symplectic Geom. {\bf 9} (2011), 343--385.

\bibitem{lejmi} M.~Lejmi, {\it Extremal almost-K\"ahler metrics},  Internat. J. Math. {\bf 21} (2010), 1639--1662.

\bibitem{LT} E.~Lerman and S.~Tolman, {\it Hamiltonian torus actions on
symplectic orbifolds and toric varieties}, Trans. Amer. Math. Soc. {\bf
349} (1997), 4201--4230.

\bibitem{LU} M.~Lejmi, M.~Upmeier, {\it Integrability theorems and conformally constant Chern scalar curvature metrics in almost Hermitian geometry}, arXiv:1703.01323.


\bibitem{marinescu-ma} X~Ma and G.~ Marinescu, {\it Holomorphic Morse inequalities and Bergman kernels}. Progress in Mathematics, 254. Birkh\"auser Verlag, Basel, 2007.


\bibitem{MSY} D.~Martelli, J. Sparks and S.-T. Yau {\it Sasaki--Einstein manifolds and volume minimisation}, Comm. Math. Phys. {\bf 280} (2008), 611--673. 


\bibitem{M}Y.~Matsushima, {\it Holomorphic vector fields on compact K\"ahler manifolds}, Conference Board of the Mathematical Sciences Regional, Conference Series in Mathematics,
No. 7. American Mathematical Society, Providence, R. I. (1971).


\bibitem{page} D.~Page, {\it A compact rotating gravitational instanton},  Phys. Lett. B {\bf 79} (1978), 235--238.

\bibitem{PD} J.~F.~Pleba\'nski and M.~Demia\'nski, {\it Rotating,
charged, and uniformly accelerating mass in general relativity},
Ann. Phys. {\bf 98} (1976), 98--127.

\bibitem{pontecorvo} M.~Pontecorvo, {\it Complex structures on Riemannian $4$-manifolds}, Math. Ann. {\bf 309} (1997), 159--177.

\bibitem{tian0} G.~Tian,  {\it On Calabi's conjecture for complex surfaces with positive first Chern class},
Invent. Math., {\bf 101} (1990), 101--172.

\bibitem{tian} G.~Tian,  {\it K\"ahler--Einstein metrics with positive scalar curvature},  Invent. Math., {\bf 137} (1997), 1--37.

\bibitem{tian-old} G.~Tian, {Bott-Chern forms and geometric stability}, Discrete Contin. Dyn. Syst. {\bf 6} (2000), 211--220.

\bibitem{Webster} S.~M.~Webster, {\it On the pseudo-conformal geometry of a K\"ahler manifold}, Math. Z.  {\bf 157} (1977), 265--270.

\bibitem{W} D.~Witt-Nystr\"om, {\it Test configurations and Okunkov bodies}, Compos. Math. {\bf 148} (2012), 1736--1756.

\bibitem{ZZ} B.~Zhou and X.~Zhu, {\it $K$-stability on toric manifolds},
Proc. Amer. Math. Soc. {\bf 136} (2008), 3301--3307.

\end{thebibliography}
\end{document}